\documentclass[11pt]{article}

\usepackage{subcaption}
\usepackage{fullpage}
\usepackage{amsmath, graphicx, amsthm, amsfonts, amssymb, amstext, mathrsfs}
\usepackage{ragged2e, enumerate, graphicx, lscape, framed}
\usepackage{pgf, tikz, float, subcaption}
\usetikzlibrary{graphs, graphs.standard, decorations.pathreplacing, calligraphy}
\usepackage[symbols,nogroupskip,sort=none]{glossaries-extra}
\usepackage{subfiles}
\usepackage[T1]{fontenc}

\newtheorem{theorem}{Theorem}[section]
\newtheorem{lemma}[theorem]{Lemma}
\newtheorem{proposition}[theorem]{Proposition}

\newtheorem{corollary}[theorem]{Corollary}
\newtheorem{conjecture}[theorem]{Conjecture}

\usepackage{hyperref}
\hypersetup{
	colorlinks=true,
    urlcolor=purple,
	linkcolor=blue,
    citecolor=purple,
}

\title{Improved Bounds for the Ultimate Independence Ratio \\ of Odd Wheels}
\author{Alexander Clow, Hitesh Kumar, Shivaramakrishna Pragada}
\date{}

\begin{document}
\maketitle
\begin{abstract}
The ultimate independence ratio of a graph $G$ is defined as $\mathscr{I}(G) = \lim_{k\rightarrow\infty } \frac{\alpha(G^{\Box k})}{|V(G)|^k},$ where $\alpha(G^{\Box k})$ is the independence number of the Cartesian product of $k$ copies of $G$. For all graphs $G$, Hahn, Hell, and Poljak (1995) proved that $\frac{1}{\chi(G)} \leq \mathscr{I}(G) \leq \frac{1}{\omega(G)}$ where $\chi(G)$ is the chromatic number, and $\omega(G)$ is the clique number of $G$. So all graphs $G$ with $\chi(G) = \omega(G)$ satisfy $\mathscr{I}(G) = \frac{1}{\chi(G)} = \frac{1}{\omega(G)}$. A construction of Zhu demonstrates that there exists a graph $G$ with $\frac{1}{\chi(G)} < \mathscr{I}(G) < \frac{1}{\omega(G)}$, so neither equality holds in general. In response, Hahn, Hell, and Poljak conjectured that all wheel graphs $W_n$ satisfy $\mathscr{I}(W_n) = \frac{1}{\chi(W_n)}$.

For even wheels $W_{2t}$ this follows from the fact $\chi(W_{2t}) = \omega(W_{2t}) = 3$. Odd wheels of length at least $5$ present a more challenging case, since $\chi(W_{2t+1}) = 4$ and $\omega(W_{2t+1}) = 3$. First, we prove that odd wheels of length at least $7$ satisfy $\mathscr{I}(W_{2t+1})\leq \frac{4t^2+6t}{3(2t+2)^2}<\frac{1}{3}$, which provides the best upper bound for large odd wheels. Next, we prove that $\mathscr{I}(W_5) \leq \frac{1019}{3888}$, improving an upper bound of Hahn, Hell, and Poljak that $\mathscr{I}(W_5) \leq \frac{11}{41}$. Our proofs combine counting arguments, recursive bounds on $\alpha(W^{\Box k}_{2t+1})$, and computer-assisted calculation in the $W_5$ case.
\end{abstract}

\section{Introduction}

\subsection{Background}
All graphs considered are simple, and we use standard graph theory terminology and notation throughout; refer \cite{west2001introduction} for any undefined terms. Given graphs $G$ and $H$,
the \textit{Cartesian product} of $G$ and $H$, written $G \Box H$,
is the graph with $V(G \Box H) = V(G) \times V(H)$
and $(u,x)(v,y) \in E(G \Box H)$ if and only if 
$u=v$ and $xy \in E(H)$ or $x=y$ and $uv \in E(G)$.
For an integer $k\geq 2$, we let
$G^{\Box k} = G^{\Box k-1} \Box G$, where we assume $G^{\Box 1} = G$. We denote the \textit{independence number, clique number, chromatic number} and \textit{fractional chromatic number} of a graph $G$ by $\alpha(G)$, $\omega(G)$, $\chi(G)$, and $\chi_f(G)$, respectively. See \cite{scheinerman2011fractional} for more on fractional chromatic number. The \textit{independence ratio} of $G$, written $i(G)$, is the relative size of a maximum independent set in $G$ compared to its order, i.e., $i(G) = \frac{\alpha(G)}{|V(G)|}$.
The \textit{ultimate independence ratio} of $G$ is then defined to be 
\[\mathscr{I}(G) = \lim_{k\rightarrow\infty } i(G^{\Box k}).\]

Investigation of this type of limiting behaviour, for different parameters and various notions of graph products, is long established in the literature, see for instance \cite{albert2001ultimate, alon2007independent, brown1996ultimate,geller1975chromatic,shannon1956zero,toth2014ultimate}. A well-known example is the notion of Shannon capacity, originally introduced in an influential paper by Shannon \cite{shannon1956zero},
which is defined by $\lim_{k\rightarrow \infty } \sqrt[k]{\alpha(G^{\boxtimes k})}$ for a given graph $G$. 
Here $\boxtimes$ denotes the strong product of graphs, refer \cite{hammack2011handbook} for the definition.

The study of ultimate independence ratio began in \cite{albertson1985homomorphisms} as a means of proving that there is no homomorphism between two graphs. Recall, a 
\textit{graph homomorphism} from a graph $G$ to a graph $H$ 
is a function $h: V(G) \rightarrow V(H)$ such that if $uv \in E(G)$, then $h(u)h(v) \in E(H)$. For a gentle introduction to graph homomorphisms, we recommend \cite{hahn1997graph}; for a broader survey, we recommend \cite{hell2004graphs}. The connection between the ultimate independence ratio and graph homomorphisms
was later established in a more general context by Zhou \cite{zhou1991chromatic,zhou1988homomorphism}. Motivated by these early results, the ultimate independence ratio was formally introduced by Hell, Yu, and Zhou \cite{hell1994independence}. They proved that for any $k$, $i(G^{\Box k+1}) \leq i(G^{\Box k})$, 
implying that $\{i(G^{\Box k})\}_{k\in \mathbb{N}}$ is a decreasing sequence of positive real numbers and so the limit $\mathscr{I}(G)$ exits. 

The first major progress in understanding the ultimate independence ratio came from Hahn, Hell, and Poljak \cite{hahn1995ultimate},
who proved the following.

\begin{theorem}[\cite{hahn1995ultimate}]
If there is a homomorphism from $G$ to $H$,
then $\mathscr{I}(H) \leq \mathscr{I}(G)$.
\end{theorem}

A valuable corollary is that for all graphs $G$,
\begin{equation}\label{eq:chi_chi_f_bound}
  \frac{1}{\chi(G)} \leq \mathscr{I}(G) \leq \frac{1}{\chi_f(G)}.  
\end{equation}
Note that $\omega(G) \le \chi_f(G) \le \chi(G)$ for any graph $G$. 
It follows that for weakly perfect graphs, i.e., graphs with $\chi = \omega$, the ultimate independence ratio is determined if either $\chi$ or $\omega$ is known.

It was a question for some time if $\mathscr{I}(G)$ is always equal to $\frac{1}{\chi(G)}$ or $\frac{1}{\chi_f(G)}$. It was shown in \cite{hahn1995ultimate}
that there exist graphs $G$ with $\frac{1}{\chi(G)} < \mathscr{I}(G) < i(G)$. This was later extended by Zhu \cite{zhu1996bounds}
who constructed a graph $G$ with $\frac{1}{\chi(G)} < \mathscr{I}(G) < \frac{1}{\chi_f(G)}$. To prove this, Zhu \cite{zhu1996bounds}
provided a new bound for the ultimate independence ratio
in terms of another colouring parameter called the star chromatic number, initially introduced by Vince \cite{vince1988star}.
Somewhat confusingly, this parameter is not the least number of vertex colours required, so that any pair of colour classes induces a forest of stars (which is also called the star chromatic number, see for instance \cite{albertson2004coloring}).

Determining the exact value of the ultimate independence ratio is, in general, a challenging problem. The smallest graph whose ultimate independence ratio is unknown is the $5$-wheel. Recall that for any integer $t\ge 1$, a $t$-\emph{wheel}, denoted $W_{t}$, is the graph obtained from the join of cycle $C_{t}$ with $K_1$.  
When $t$ is even, $\chi(W_t) = \omega(W_t) = 3$, and so for all $t$, $\mathscr{I}(W_{2t}) = \frac{1}{3}$.
Odd wheels are more interesting, which led Hahn, Hell, and Poljak \cite{hahn1995ultimate} to conjecture the following. Since $\chi(W_{2t+1}) = 4$, note here that $\mathscr{I}(W_{2t+1})\ge \frac{1}{4}$ by \eqref{eq:chi_chi_f_bound}.

\begin{conjecture}[\cite{hahn1995ultimate}]
\label{Conj: Odd Wheels}
    For all integers $t\geq 1$, 
    \[
    \mathscr{I}(W_{2t+1}) = \frac{1}{4}.
    \]
\end{conjecture}

This 30-year-old conjecture, in particular, the special case of determining $\mathscr{I}(W_5)$, was recently resurrected by Hahn in a talk\footnote{Hahn's talk was titled \emph{Resurrection – revisiting old problems}.} at the Canadian Math Society winter meeting in 2024. The best known progress on Conjecture~\ref{Conj: Odd Wheels}
comes in \cite{hahn1995ultimate} where it was shown that 
\begin{equation}\label{eq:Wheel_5_Hahn_bound}
  \mathscr{I}(W_{5}) \leq \frac{11}{41}.  
\end{equation}
This bound was obtained by proving that $\chi_f(W_5^{\Box 2}) = \frac{41}{11}$.

For $t\geq 3$, the best known upper bound for the ultimate independence ratio of $W_{2t+1}$ is
\begin{equation}\label{eq:higher_wheels_known_bound}
\mathscr{I}(W_{2t+1})\leq \frac{t}{3t+1}.    
\end{equation}
The bound in \eqref{eq:higher_wheels_known_bound} has never been observed before, but it follows from known results in the literature. Precisely, the bound in \eqref{eq:higher_wheels_known_bound} follows using \eqref{eq:chi_chi_f_bound} and the fact that $\chi_f(W_{2t+1}) = 3 + \frac{1}{t}$. This fact is true because the fractional chromatic number of a join $G \lor H$ satisfies 
$\chi_f(G \lor H) = \chi_f(G)+ \chi_f(H)$, 
see \cite{johnson2009fractional},
while it is well known that $\chi_f(C_{2k+1}) = 2 + \frac{1}{k}$, see \cite[page 31]{scheinerman2011fractional}.

\subsection{Our Results}

We improve the upper bounds on 
$\mathscr{I}(W_{2t+1})$ for all $t\geq 2$, thus making the first progress towards 
Conjecture~\ref{Conj: Odd Wheels} in three decades. We note that any progress here is of interest, 
since estimating $\mathscr{I}(W_{2t+1})$ up to an arbitrary error is algorithmically impractical using known methods.

Let $\ell$ be a fixed positive integer, preferably small. Clearly, for a graph $G$,  $\mathscr{I}(G)\le i(G^{\Box \ell})$, and so if one can compute $\alpha(G^{\Box \ell})$, then it would lead to an upper bound on $\mathscr{I}(G)$. When the graph is $W_{2t+1}$, we know the value of $\alpha(W_{2t+1}^{\Box 2})$ (see Proposition \ref{prop:higher_wheel_second_power}), but it is difficult to compute $\alpha(W_{2t+1}^{\Box 3})$ for large $t$. For some small values of $t$, $\alpha(W_{2t+1}^{\Box 3})$ is known (see Table \ref{table:smallwheel}), but computing even $\alpha(W_{5}^{\Box 4})$ is difficult using our machines. 
To overcome this, 
we prove a new upper bound for $\mathscr{I}(G)$.
Notice that this generalizes a result by  Hahn, Hell, and Poljak
regarding vertex transitive graphs, see Theorem~4.1 in \cite{hahn1995ultimate}.

\begin{theorem}\label{thm:upper_bound_triangle}
Let $k, \ell \ge 1$ be fixed integers.
If $G$ is a graph containing a clique of order $k$,
then  
\[ \mathscr{I}(G) \le \frac{\alpha(G^{\Box \ell}\Box K_k)}{k|V(G)|^{\ell}}.\]
Moreover,
\[ \mathscr{I}(G) = \lim_{\ell \rightarrow \infty}\frac{\alpha(G^{\Box \ell}\Box K_k)}{k|V(G)|^{\ell}}.\]
\end{theorem}

Applying Theorem~\ref{thm:upper_bound_triangle} to
$W_{2t+1}$
with $k=3$ gives the following corollary.
Since we are focused on odd wheels in this paper, this corollary is the principle application of Theorem~\ref{thm:upper_bound_triangle} we will use.

\begin{corollary}
Let $t\ge 2$ and $\ell \ge 1$ be fixed integers. Then  
\[ \mathscr{I}(W_{2t+1}) \le \frac{\alpha(W_{2t+1}^{\Box \ell}\Box K_3)}{3(2t+2)^{\ell}}.\]
\end{corollary}

We first use Theorem~\ref{thm:upper_bound_triangle} to obtain a non-trivial upper bound on $\mathscr{I}(W_{2t+1})$ for any $t$ without the need for computer assistance. Observe that for any graph $G$, $\alpha(G\Box K_k) \leq |V(G)|$
with equality if and only if $G$ is $k$-colourable. Since each odd-wheel is a $4$-critical graph, the $\ell=1$ case of Theorem~\ref{thm:upper_bound_triangle} immediately implies the following: for all $t\geq 2$,

\begin{equation}\label{eq:all_odd_l_1}
    \mathscr{I}(W_{2t+1}) \leq \frac{2t+1}{6t+6}.
\end{equation}

It is easily seen that the upper bound in \eqref{eq:all_odd_l_1} is already better than the bound in \eqref{eq:higher_wheels_known_bound} for any $t \ge 2$.
We improve this even further for all $t\ge 3$ as follows.
Assuming $t\ge 3$ here is not significant, since we bound $\mathscr{I}(W_5)$ by different methods later.

\begin{theorem}\label{thm:higher_wheel_best_bound}
    For all $t\ge 3$, we have
    \[ \mathscr{I}(W_{2t+1}) \le \frac{4t^2 + 6t}{3(2t+2)^2}.\]
\end{theorem}

To prove Theorem \ref{thm:higher_wheel_best_bound}, we first investigate independent sets in $W_{2t+1}^{\Box 2}$ (Lemma \ref{lemma:ind_set_higher_power_large_center}), which also leads to a proof of the fact that $\alpha(W_{2t+1}^{\Box 2}) = t(2t+1) + 1$ (Proposition \ref{prop:higher_wheel_second_power}). We then estimate $\alpha(W_{2t+1}^{\Box 2}\Box K_3)$ and use Theorem \ref{thm:upper_bound_triangle} to finish the proof of Theorem \ref{thm:higher_wheel_best_bound}.  

We believe that examining independent sets in $W_{2t+1}^{\Box 2}\Box K_3$ can lead to further insights. So for $2 \leq t \leq 5$,
we include a figure that shows a maximum independent set in $W_{2t+1}^{\Box 2} \Box K_3$. These figures can be found in Appendix~A.

The rest of our work focuses on small odd wheels, with a particular focus on $W_5$. We believe $W_5$ is the most interesting open case of Conjecture \ref{Conj: Odd Wheels}, and the resolution of this case can potentially lead to a proof of the full conjecture. One possibility to improve the bound in \eqref{eq:Wheel_5_Hahn_bound} for $\mathscr{I}(W_5)$ is by computing $\alpha(W^{\Box4}_5)$, but this was not tractable using our machines. 
Instead, we exactly determine the independence number of $W_{5}^{\Box 3} \Box K_3$, which was previously unknown.

\begin{theorem}\label{thm:5_wheel_thirdpower_triangle}
    $\alpha(W_{5}^{\Box 3}\Box K_3) = 170$.
\end{theorem}

Our proof of Theorem \ref{thm:5_wheel_thirdpower_triangle} is computer-assisted. We first identify a maximum independent set of size 170 (see Figure~\ref{fig:independent_set_170}), which gives the lower bound. We then use structural insights to identify different possibilities for a large independent set in $W_5^{\Box 3}$ and then use integer programming, in combination with structural arguments, to show that there are no independent sets of size more than 170. 

An immediate corollary of Theorem~\ref{thm:upper_bound_triangle} and Theorem~\ref{thm:5_wheel_thirdpower_triangle}
is the following:
\begin{equation}
    \mathscr{I}(W_5) \le \frac{85}{324},
\end{equation}
which is already better than \eqref{eq:Wheel_5_Hahn_bound}. We take further steps to improve this bound.
Using Theorem \ref{thm:5_wheel_thirdpower_triangle} we show that $\alpha(W^{\Box 4}_5) \le 343$ (Theorem \ref{thm:5_wheel_fourth_power}). Additionally, we estimate $\alpha(W_5^{\Box 4}\Box K_3)$ (Theorem~\ref{thm:5_wheel_fourthpower_triangle}). Using Theorem~\ref{thm:upper_bound_triangle}, this estimate implies the following. 

\begin{theorem}\label{thm:5_wheel_best_bound} $\mathscr{I}(W_5) \le \frac{1019}{3888}$.
\end{theorem}

For other small odd wheels, we obtain improvements on Theorem~\ref{thm:higher_wheel_best_bound}, purely using computer assistance. Our results are summarized in Table~\ref{table:smallwheel}. All of the reported computations were conducted on our personal laptop computers.

\renewcommand{\arraystretch}{1.5}
\begin{table}[!h]
\centering
\begin{tabular}{| c | c | c | c | c |} 
\hline \hline
$2t+1$ & $\alpha(W_{2t+1}^{\Box 3})$ & $\chi_f(W_{2t+1}^{\Box 2})$ &  $\alpha(W^{\Box 2}_{2t+1} \Box K_3)$ & $\mathscr{I}(W_{2t+1})\leq $  \\ \hline \hline 
$5$ & $58$ & $41/11$ \cite{hahn1995ultimate} & $29$ & $\frac{1019}{3888}$ \\ \hline
$7$ & $156$ & $39/11$ & $54$ & $9/32$ \\ \hline
$9$ & $336$ & $127/37$ & $87$ & $29/100$ \\ \hline
$11$ & $620$ & $47/14$ & $128$ & $8/27$ \\ \hline
$13$ & $1032$ & $261/79$ & $177$ & $59/196$ \\ \hline
$15$ & ? & $173/53$ & $234$ & $39/128$ \\ \hline \hline
\end{tabular}
\caption{Best known computational results. See our code  \href{https://github.com/Shivaramkratos/Ultimate_Independence_ratio_Python_Gurobi_code}{here}. The upper bound for $\mathscr{I}(W_{2t+1})$ is obtained by using the values for $\alpha(W^{\Box 2}_{2t+1}\Box K_3)$ and Theorem~\ref{thm:upper_bound_triangle}. Fractional chromatic numbers are computed using methods described in Section~5.
}
\label{table:smallwheel}
\end{table}

\subsection{Organization of the Paper}

This paper is organized as follows. 
In Section~\ref{sec:Notation}, we introduce notation and results from the literature that will be useful later in the paper.
Our contributions begin in Section~3 and conclude in Section~6.
This part of the paper can be divided into two parts, 
Sections~3 and 4, which deal with purely theoretical results,
and Section~5 and 6, which involve computer assistance.
Section~3 is used to prove Theorem~\ref{thm:upper_bound_triangle} while 
Section~4 studies the structure of $W^{\Box 2}_{2t+1}$ with the goal of proving Theorem~\ref{thm:higher_wheel_best_bound}.
In Section~5 we outline our computational methods and how we will rigorously
present our computer-assisted proofs.
Finally, Section~6 is spent applying our theoretical techniques, together with computer assistance,
to obtain improved bounds on the ultimate independence ratio of $W_5$.
We conclude with a section discussing future work, and an appendix is included, which displays some helpful figures for understanding maximum independent sets in $W^{\Box2}_{2t+1}\Box K_3$ when $t$ is small.

\section{Notation and Preliminaries}\label{sec:Notation}

We first set up the notation that will be used throughout this paper. For a fixed $t\ge 2$, let 
\[V(W_{2t+1}) = \{w_i: i \in \mathbb{Z}_{2t+1}\}\cup \{w_*\}\] where $w_*$ denotes the dominating vertex in $W_{2t+1}$. Here $\mathbb{Z}_{2t+1}$ denotes the cyclic group of order $2t+1$. This is helpful in the indexing of the vertices. Then
\[ V(W_{2t+1}^{\Box k}) = \{(a_1, \ldots, a_k): a_i \in V(W_{2t+1})\}.\]
We define $H_{p}^{(k)}$ to be the subgraph of $W_{2t+1}^{\Box k}$ induced by the vertices 
\[ \{(a_1, \ldots, a_{k-1}, w_p): a_i\in V(W_{2t+1}), 1\le i\le k-1\}.\]
Clearly, $H_p^{(k)}\cong W_{2t+1}^{\Box k-1}$. 
We define $H_{p,q}^{(k)}$ to be the subgraph of $W_{2t+1}^{\Box k}$ induced by the vertices 
\[\{(a_1, \ldots, a_{k-2},w_q, w_p): a_i\in V(W_{2t+1}), 1\le i\le k-2\}.\]
Note that $H_{p,q}^{(k)}\cong W_{2t+1}^{\Box k-2}$. Also, $V(H_{p,q}^{(k)})\subset V(H_p^{(k)})$. 

Next, we recall a known result that is implicitly used later.

\begin{proposition}
Let $G$ and $H$ be graphs. Then 
    \[ \alpha(G\Box H)\le \min\{|V(G)|\alpha(H) , |V(H)|\alpha(G)\}.\]
\end{proposition}

\section{A General Bound for Ultimate Independence}

In this short section, we provide a proof of 
Theorem~\ref{thm:upper_bound_triangle}. 
Our argument proceeds by using counting techniques to obtain a recurrence relation that upper bounds 
$i(G^{\Box \ell+1})$.
From here we use the limit in the definition of $\mathscr{I}(G)$,
together with this recurrence relation,
to obtain the desired bound.
We note that our argument does not imply that 
the bound in 
Theorem~\ref{thm:upper_bound_triangle}
is true when $i(G^{\Box q})$ replaces $\mathscr{I}(G)$,
for any choice of $q$.
We simply guarantee the inequality in the limit.

\begin{proof}[Proof of Theorem~\ref{thm:upper_bound_triangle}]
Suppose that $G$ has $n$ vertices.
Let $I$ be a maximum independent set in $G^{\Box p+1}$. 
By assumption, there exists a set clique $K$ of order $k$ in $G$.
For each $v \in V(G)$, define $I_v = I\cap \{(u_1,\dots, u_p, v): (u_1,\dots, u_{p}) \in V(G^{\Box p})\}$.
Then
\begin{align*}
    |I| & = \sum_{v \in V(G)} I_v \\
    & = \left( \sum_{v \in V(K)} I_{v} \right) + \left( \sum_{u \notin V(K)} I_u\right) \\
    & \le \alpha(G^{\Box p} \Box K_k) + \left(n - k\right)\alpha(G^{\Box p}).
\end{align*}
Thus,
\begin{align*}
    \alpha(G^{\Box p+1}) & \le \alpha(G^{\Box p} \Box K_k) + \left(n - k\right)\alpha(G^{\Box p}) \\
    & \leq n^{p-\ell}\alpha(G^{\Box \ell} \Box K_k) + \left(n - k\right)\alpha(G^{\Box p}) 
\end{align*}
Using the above recurrence, we see that 
\begin{align*}
    i(G^{\Box p+1}) & =  \frac{\alpha(G^{\Box p+1})}{n^{p+1}} \\
    & \le \frac{\alpha(G^{\Box \ell} \Box K_k)}{n^{\ell+1}} \left( \sum_{i=0}^{p-\ell} \left(\frac{n-k}{n} \right)^i \right) + \frac{(n-k)^{p-\ell + 1}\alpha(G^{\Box \ell})}{n^{p+1}}
\end{align*}
Recall that $\ell$ is a constant.
By applying the definition of $\mathscr{I}(G)$ and taking limits we get
\begin{align*}\label{eq:upper_bound_with_triangle}
    \mathscr{I}(G) & = \lim_{p\rightarrow \infty} i(G^{\Box p}) 
    \\
    & = \lim_{p\rightarrow \infty} i(G^{\Box p+1}) 
    \\
    & \le \lim_{p\rightarrow \infty} \left(\frac{\alpha(G^{\Box \ell} \Box K_k)}{n^{\ell+1}} \left( \sum_{i=0}^{p-\ell} \left(\frac{n-k}{n} \right)^i \right) + \frac{(n-k)^{p-\ell + 1}\alpha(G^{\Box \ell})}{n^{p+1}} \right)
    \\
    & = \frac{\alpha(G^{\Box \ell} \Box K_k)}{n^{\ell+1}} \left( \lim_{p\rightarrow \infty} \left(  \sum_{i=0}^{p-\ell} \left(\frac{n-k}{n} \right)^i \right) \right)  + \lim_{p\rightarrow \infty}\Bigg(
    \frac{(n-k)^{p-\ell + 1}\alpha(G^{\Box \ell})}{n^{p+1}}
    \Bigg)
    \\
    & =  \frac{\alpha(G^{\Box \ell} \Box K_k)}{n^{\ell+1}}\left( \frac{n}{k} \right) + 0
    \\
    & = \frac{\alpha(G^{\Box \ell} \Box K_k)}{kn^{\ell}}.
\end{align*}
We conclude that $\mathscr{I}(G) \leq \frac{\alpha(G^{\Box \ell} \Box K_k)}{kn^{\ell}}$ as desired.

Now we turn our attention to proving 
\[
\mathscr{I}(G) = \lim_{\ell\rightarrow \infty} \frac{\alpha(G^{\Box \ell} \Box K_k)}{kn^{\ell}}.
\]
Let $L$ denote the value of the limit on the right-hand side.
For all $\ell\geq 1$, 
the first part of the proof implies
$\mathscr{I}(G) \leq \frac{\alpha(G^{\Box \ell} \Box K_k)}{kn^{\ell}}$.
Hence, $L \ge \mathscr{I}(G)$.
It remains to be shown that $L \le \mathscr{I}(G)$.
To start,
observe that for all $\ell\ge 1$,
\begin{align*}
    \frac{\alpha(G^{\Box \ell} \Box K_k)}{kn^{\ell}}  \le \frac{k\alpha(G^{\Box \ell})}{kn^\ell}  = \frac{\alpha(G^{\Box \ell})}{n^\ell}  = i(G^{\Box \ell}).
\end{align*}
Hence, 
\begin{align*}
    \lim_{\ell\rightarrow \infty} \frac{\alpha(G^{\Box \ell} \Box K_k)}{kn^{\ell}}  \le \lim_{\ell \rightarrow\infty} i(G^{\Box \ell}) 
     = \mathscr{I}(G).
\end{align*}
This completes the proof.
\end{proof}

\section{On Large Odd Wheels}

This section is spent leveraging Theorem~\ref{thm:upper_bound_triangle} to obtain a non-trivial upper bound on $\mathscr{I}(W_{2t+1})$ for all $t$.
Specifically, we prove $\alpha(W^{\Box 2}_{2t+1} \Box K_3) \leq 4t^2 +6t+1$.
Before bounding the size independent sets in $W^{\Box 2}_{2t+1} \Box K_3$
we examine independent sets in $W^{\Box 2}_{2t+1}$.
We first note some easy facts.

\begin{proposition}\label{prop:higher_wheels_small_independence} For $t\ge 2$, we have 
    \begin{enumerate}[$(i)$]
        \item $\alpha(W_{2t+1}) = t.$
        \item $\alpha(W_{2t+1}\Box K_3) = 2t+1$. 
    \end{enumerate}
\end{proposition}

\begin{proof}
    Assertion $(i)$ is trivial. We give a proof for $(ii)$. Clearly,  
    given $\chi(W_{2t+1}) = 4$,
    \[ \alpha(W_{2t+1}\Box K_3)\le 2t+1.\]
    Now, assume that $V(K_3) = \{w_i: i\in  \mathbb{Z}_3\}$. Then
    \[
    \{(w_{2i+j},w_j): 0 \leq i\leq t-1, 0\leq j \leq 1\} \cup \{(*,2)\}
    \]
    is an independent set of size $2t+1$ in $W_{2t+1}\Box K_3$. 
    Hence, $\alpha(W^{\Box 2}_{2t+1})\geq 2t+1$ completing the proof.
\end{proof}

Next, we will determine $\alpha(W_{2t+1}^{\Box 2})$. We first prove a more general lemma about independent sets in $W_{2t+1}^{\Box 2}$. Let $I$ denote an independent set in $W_{2t+1}^2$, and for $i\in \mathbb{Z}_{2t+1}\cup \{*\}$ define 
\[ I_i = I \cap V(H_i^{(2)}) =\{(a,w_i)\in I: a \in V(W_{2t+1})\}.\]

\begin{lemma}\label{lemma:ind_set_higher_power_large_center}
Suppose $|I_*\backslash\{(w_*, w_*)\}| = k$. Then 
\[|I| \le t(2t+1) + 1 + (1-t)k.\]
\end{lemma}

\begin{proof} If $k = 0$, then 
\[|I|\le \alpha(W_{2t+1}\Box C_{2t+1}) + 1\le \alpha(W_{2t+1})\cdot (2t+1) + 1 = t(2t+1) + 1,\]
and the assertion holds.

Now, assume that $k\ge 1$. Observe that if $(a, w_*)\in I_*$, then $(a, w_i)\notin I_i$ for all $i\in \mathbb{Z}_{2t+1}$. For any $i$, the set $I_i - \{(b,w_i): (b, w_*)\in I_*\}$ induces a subgraph in $W_{2t+1}^{\Box 2}$ which is isomorphic to: either $H_1:=\cup_{j = 1}^{k} P_{s_j}$ if $(w_*, w_*)\in I_*$, or $H_2:=\left(\cup_{j = 1}^{k} P_{s_j}\right) \vee K_1$ if $(w_*, w_*)\notin I_*$. Here $P_{s_j}$ is a path on $s_j$ vertices, and we have $\sum_{j=1}^{k}s_j = 2t + 1 - k$. Note that $|V(H_1)| = 2t + 1-k\ge 1$ since $k \le t$. It follows that $\alpha(H_1) = \alpha(H_2)$ and so, irrespective of whether $(w_*, w_*)\in I_*$, we have 
\begin{align*}
    |I| & \le |I_*| + \alpha(\cup_{j = 1}^{k} P_{s_j} \Box C_{2t+1})\\
    & = |I_*| + \sum_{j=1}^k \alpha(P_{s_j}\Box C_{2t+1})\\
    & \le (k + 1) + \sum_{j=1}^{k}ts_j\\
    & = k + 1 + t(2t + 1 -k), 
\end{align*}
completing the proof.
\end{proof}

We can now determine $\alpha(W_{2t+1}^{\Box 2})$. 

\begin{proposition}\label{prop:higher_wheel_second_power}
    For any $t\ge 2$, we have 
    \[ 
    \alpha(W_{2t+1}^{\Box 2}) =  (2t+1)t+1.
    \]
\end{proposition}

\begin{proof}
     The reader can verify that 
    \[
    \{(w_{2i+j},w_j): 0 \leq i\leq t-1, 0\leq j \leq 2t\} \cup \{(*,*)\}
    \]
    is an independent set of size $(2t+1)t+1$ in $W_{2t+1}^{\Box 2}$. 
    Now, if there exists an independent set $I$ with $|I|> (2t+1)t + 1$, then $|I_*|\ge 2$. 
    Notice that $(w_*,w_*)\in I_*$ implies $|I_*| = 1$, so we conclude that $(w_*,w_*)\notin I_*$.
    Thus $k = |I_*\backslash \{(w_*, w_*)\}| = |I_*| \ge 2$. 
    Using Lemma \ref{lemma:ind_set_higher_power_large_center}, we have 
    \[(2t+1)t + 1 < |I|\le t(2t+1) + 1 + (1-t)k < t(2t+1) + 1, \]
    given $t,k\geq 2$. This is a contradiction.
\end{proof}

In what follows, we estimate $\alpha(W_{2t+1}^{\Box 2}\Box K_3)$ and use it to furnish a proof of Theorem \ref{thm:higher_wheel_best_bound}. 

\begin{proof}[Proof of Theorem \ref{thm:higher_wheel_best_bound}] Assume $V(K_3) = \{w_i: i\in \mathbb{Z}_3\}$.
Let $S$ denote an independent set in $W_{2t+1}^{\Box 2}\Box K_3$. For $i\in \mathbb{Z}_3$ and $j\in \mathbb{Z}_{2t+1}\cup \{*\}$ define  
\[S_i = \{(a_1, a_2, w_i) \in S: a_1, a_2 \in V(W_{2t+1})\}\quad \text{and}\quad S_{i, j} =  \{(a_1, w_j, w_i) \in S: a_1\in V(W_{2t+1})\}.\]
Let $k_i = |S_{i,*}\backslash \{(w_*, w_*, w_i)\}|$ for $i\in \mathbb{Z}_3$. Using Lemma \ref{lemma:ind_set_higher_power_large_center}, we see that for any $i$, 
\begin{align*}
|S_i|  & \le t(2t+1) + 1 + (1-t)k_i, 
\end{align*}
which implies
\begin{align*}
    k_i \le \frac{t(2t+1) + 1 - |S_i|}{t-1}.
\end{align*}
Observe that $\sum_{i=0}^2 |S_{i,*}| \le k_1+k_2+k_3+1$, with equality only if there exists an $i$ such that $k_i = 0$.
This is because
there is at most one index $i$ such that $(w_*, w_*, w_i)\in S$, 
while $(w_*, w_*, w_i)\in S$ implies that $S_{i,*} = \{(w_*, w_*, w_i)\}$.
Suppose without loss of generality that $k_1\ge k_2 \ge k_3$.
Hence, $\sum_{i=0}^2 |S_{i,*}| \le \max\{ k_1+k_2+k_3, k_1+k_2+1\}$.
Now, using Proposition \ref{prop:higher_wheels_small_independence}
\begin{align*}
    |S|& \le \left(\sum_{i=0}^2 |S_{i,*}|\right) + \sum_{j=0}^{2t}\left(|S_{0,j}| + |S_{1,j}| + |S_{2,j}| \right)\\
    & \le \max\{ k_1+k_2+k_3, k_1+k_2+1\} + (2t+1)\alpha(W_{2t+1}\Box K_3)\\
    & \le \max\left\{ 1 + \sum_{i=0}^1 \left(\frac{t(2t+1)  + 1 - |S_i|}{t-1}\right), \sum_{i=0}^2 \left(\frac{t(2t+1)  + 1 - |S_i|}{t-1}\right)\right\} + (2t+1)^2\\
    & = \max\left\{ 1 + \sum_{i=0}^1 \left(\frac{t(2t+1) + 1}{t-1}\right),  \sum_{i=0}^2 \left(\frac{t(2t+1) + 1}{t-1}\right) \right \}  + (2t+1)^2 - \frac{|S|}{t-1}\\
\end{align*}
Rearranging the terms, we get 
\begin{align*}
    |S| \le \frac{(t-1)\left(1 + (2t+1)^2\right) + 2t(2t+1) + 2}{t} = 4t^2+4t
\end{align*}
if the first term is the maximum, or we get
\begin{align*}
    |S| \le \frac{(t-1)(2t+1)^2 + 3t(2t+1) + 3}{t} = 4t^2 + 6t + \frac{2}{t}.
\end{align*}
if the second term is the maximum.
Since $|S|$ is an integer, and $t \ge 3$ we get that 
\[|S| \leq 4t^2 + 6t.\]
Finally, using Theorem \ref{thm:upper_bound_triangle} with $\ell = 2$, we get 
\[ \mathscr{I}(W_{2t+1}) \le \frac{4t^2 + 6t}{3(2t+2)^2}.\]
This completes the proof.
\end{proof}

\section{Computational Framework}

The purpose of this section is to outline the computational methods we used. This is to help the reader verify our computational results and potentially build on our work to obtain new bounds.
All of our code can be found \href{https://github.com/Shivaramkratos/Ultimate_Independence_ratio_Python_Gurobi_code}{here}.
Please see the \href{https://github.com/Shivaramkratos/Ultimate_Independence_ratio_Python_Gurobi_code/blob/main/ReadMe.txt}{ReadMe.txt} file for instructions on how to read,
run, and reproduce our code.

This section is divided into two subsections.
Subsection \ref{subsection:indendendent_sets} covers how we compute independence numbers,
and how we use integer linear programming methods to disprove the existence of certain kinds of independent sets
of near maximum size.
In Subsection \ref{subsection:indendendent_sets} we describe how we compute the fractional chromatic numbers
appearing in Table~\ref{table:smallwheel}.

\subsection{Understanding Independent Sets}\label{subsection:indendendent_sets}

Here, we outline the integer linear programming (IP) formulations that we use for calculations related to the existence of independent sets.
This goes beyond computing the independence number, as we are interested not only in the size of independent sets,
but their structure.
Our calculations are run using Gurobi's linear programming solvers.

For a graph $G$, define the matrix $B(G) = [b_{ue}] \in \{0,1\}^{V(G)\times E(G)}$, where $b_{ue} = 1$ if $u$ is an endpoint of the edge $e$, and $b_{ue} = 0$ otherwise. The matrix $B(G)$ is called the \textit{incidence matrix} of $G$. 
Let $\mathbf{1}$ denote the all-ones vector of appropriate order. The standard IP formulation for the independence number $\alpha(G)$ is given by:

\begin{align}\label{eq: standardAlpha}
    \max \quad & \sum_{v \in V(G)}x_{v} \\
    \text{subject to} \quad &  B(G)^Tx \leq \mathbf{1} \nonumber\\
            & x \in \{0,1\}^{V(G)}. \nonumber
\end{align}


Notice here that for any independent set $S$ in the graph $G$, the indicator vector for $S$ is an element of the feasible region of this IP. Thus, if we wish to study independent sets of $G$ satisfying additional properties $\mathcal{P}$, it may be sufficient to add constraints to this program, depending on the choice of $\mathcal{P}$.
In particular, if one wishes to show that there is no independent set of $G$ with a certain property $\mathcal{P}$, then it is sufficient to show that a version of this program with added constraints is infeasible.

Using the simple formulation above, the number of variables and constraints grows quickly in higher powers of Cartesian products. 
To mitigate this, we use the structure of the Cartesian product to obtain finer control over the structure of independent sets.
For simple graphs $G$ and $H$, the following IP gives $\alpha(G \Box H)$: 
\begin{align}\label{prog:alpha_QCIP}
\max \quad & \sum_{v \in V(H)} \mathbf{1}^Tx_{v} \\
\text{subject to} \quad & B(G)^Tx_v \leq \mathbf{1}  & &\hspace*{-3cm}\forall \ v\in V(H)\nonumber\\
& x_u + x_v \leq \mathbf{1} & &  \hspace*{-3cm} \forall \ uv \in E(H) \nonumber\\
& x_v \in \{0,1\}^{|V(G)|} & & \hspace*{-3cm} \forall \ v\in V(H).\nonumber 
\end{align}

This is a straightforward generalization of the standard program for independence number. The finer control this gives us to examine the structure of independent sets satisfying additional properties is worth stating explicitly.

Note here that $V(G \Box H)$ has a natural partition of the form $\{\{(u,v): u \in V(G)\}: v \in V(H)\}$,
where the subgraph induced by each of these sets is isomorphic to $G$.
Let $I$ be an independent set in $G \Box H$.
In Program~\ref{prog:alpha_QCIP} each variable $x_v$ is the indicator vector for the subset 
\[I_v = I\cap \{(u,v): u \in V(G)\}\] of $I$.
Observing this, Program~\ref{prog:alpha_QCIP} provides the following advantage.
Recalling how Lemma~\ref{lemma:ind_set_higher_power_large_center} is used to prove Theorem~\ref{thm:higher_wheel_best_bound}, it is sometimes useful to understand how forcing a subgraph to contain many vertices in an independent set
affects the structure of the independent set in the rest of the graph.
By introducing the variable vectors $x_v$, rather than giving each vertex of the graph $G\Box H$ its own variable, there is an easy-to-implement and intuitive way to study independent sets $I$ where 
$\{(u,v): u \in V(G)\}$ contains at least $k$ vertices of $I$.
To do so, one can simply add the constraint $\mathbf{1}^Tx_v \geq k$ to  Program~\ref{prog:alpha_QCIP}.
Of course, one can add several such constraints simultaneously as well.

In Lemma~\ref{lemma:third_power_center_independence}, 
Lemma~\ref{lemma:57 center 9}, and Lemma~\ref{lemma:thirdpower_wit5h_K3_infeasible_sets} we do exactly this to prove certain kinds of independent sets do not exist in $W^{\Box 3}_5$.
That is, we add additional constraints of the form $\mathbf{1}^Tx_v \geq k$ or $\mathbf{1}^Tx_v = k$ to Program~\ref{prog:alpha_QCIP},
then run the resulting program in Gurobi, which certifies the modified program is infeasible,
implying that no independent sets with the prescribed structure exist.
Leveraging the non-existence of independent sets with these structures, we then prove theoretically 
that $\alpha(W^{\Box 3}_5) = 170$, $\alpha(W_5^{\Box 4}) \leq 343$, and $\alpha(W^{\Box 4}_5 \Box K_3) \leq 1019$, respectively.

Each time that we conduct such a procedure, we must run a different integer program in Gurobi.
For each instance of such a calculation, we provide a distinct file, linked in the text, so that the reader 
can reproduce and verify our computational work.
All files are commented to assist the reader's understanding
and files may be read in any order.

\subsection{Calculating Fractional Chromatic Number}\label{subsection:fractional_chromatic_no} 

Hahn, Hell, and Poljak~\cite{hahn1995ultimate} proved the following linear programming characterization of the fractional chromatic number $\chi_f(G)$ of a graph $G$.  Let $\mathcal{I}(G)$ denote the family of all independent sets of $G$. Then the following is a program for $\frac{1}{\chi_f(G)}$:
\begin{align}\label{prog:1/chi_f}
    \min \quad &  z \\
    \text{subject to} \quad & \sum_{v \in V(G)} x_{v} = 1  \nonumber\\
    & \sum_{v \in I} x_{v} \le z & & \hspace*{-4cm} \forall \, I \in \mathcal{I}(G) \nonumber\\
    & x_v \ge 0 & &\hspace*{-4cm} \forall \, v \in V(G).\nonumber
\end{align}


It is straightforward to observe that it suffices to impose the constraints 
\[ \sum_{v\in I} x_v \le z \] only for the \emph{maximal} independent sets $I$ of $G$, any non-maximal independent set is contained in a maximal one, and therefore its corresponding constraint is dominated.

To simplify the program further, we exploit the symmetries of the graph via its automorphism group. 

\begin{theorem}\label{thm:X_f_automorphism_orbits} 
Let $G=(V, E)$ be a simple graph, and let $\mathcal I(G)$ denote the family of independent sets of $G$. Let $\Gamma=\mathrm{Aut}(G)$ act on $V$ and on $\mathcal I$ in the natural way. Then there exists an optimal solution $x$ to Program \ref{prog:1/chi_f} that is constant on the orbits of maximal independent sets under $\Gamma$. 
\end{theorem} 

\begin{proof} 
Let $y$ be any feasible solution of the Program \ref{prog:1/chi_f}. For each $g\in\Gamma$ define another feasible solution $g\cdot y$ by $(g\cdot y)_v := y_{g^{-1}(v)}$. Since $g$ permutes vertices and maximal independent sets, $g\cdot y$ satisfies all constraints of the Program \ref{prog:1/chi_f} and has the same objective value as $y$. 

Now consider the following feasible solution: 
\[ x := \frac{1}{|\Gamma|}\sum_{g\in\Gamma} g\cdot y. \] 
Then $x$ is feasible, has the same objective value as $y$, and is invariant under $\Gamma$, that is to say, for all $h\in\Gamma$ we have $h\cdot x=x$. Thus $x$ is constant on the orbits of $V$ under $\Gamma$ and we are done. 
\end{proof}

The automorphism group partitions the vertex set into the following orbits:
\[
T_1=\{(w_*,w_*)\},\quad
T_2=\{(w_i,w_*),(w_*,w_i):0\le i\le 2t\},\quad
T_3=\{(w_i,w_j):0\le i,j\le 2t\}.
\]
By Theorem~\ref{thm:X_f_automorphism_orbits}, the LP variables may be taken to be constant on these orbits, and each constraint depends only on the number of vertices from each orbit appearing in an independent set.

Following the notation of~\cite{hahn1995ultimate}, we associate to every independent set $I\in\mathcal{I}(W_{2t+1}^{\Box 2})$ its profile
\[
P_I=(p_1,p_2,p_3),
\]
where $p_i$ is the number of vertices of $I$ lying in $T_i$.  
We call a profile $(p_1,p_2,p_3)$ \emph{maximal} if there is no other profile $(q_1,q_2,q_3)$ such that 
$q_i\geq p_i$ for each $i\in\{1,2,3\}$ and 
\[
q_1+q_2+q_3 \;\ge\; p_1+p_2+p_3.
\]
Maximal profiles correspond exactly to those independent sets whose orbit–counts cannot be simultaneously increased, hence they yield all potentially tight constraints in the LP. 

The maximal profiles for $W_5^{\Box 2}$ were obtained in \cite{hahn1995ultimate}, and they are the following:
\[(1,0,10), (0,2,8), (0,3,6), (0,4,5).\] 
Using this, they set up the reduced LP and solved it to obtain $\chi_f(W_5^{\Box 2}) = \frac{41}{11}$, which in turn yields a better upper bound on $\mathscr{I}(W_5)$. For higher odd wheels, finding the profiles manually is quite challenging, so we used computer assistance to do so. A naive method to search for the profiles is to enumerate over all maximal independent sets and count the number of vertices that belong to each orbit, but the number of maximal independent sets grows exponentially with respect to the order of the wheel. For example, $W^{\Box2}_7$ has $909874$ (approximately a million) maximal independent sets as compared to $\chi_f(W^{\Box2}_5)$, which has only $2770$ maximal independent sets. To simplify the computational effort, we need the following two observations:

\begin{enumerate}
    \item The maximum independent set in $W_{2t+1}^{\Box 2}$ by Lemma \ref{lemma:ind_set_higher_power_large_center} and Proposition \ref{prop:higher_wheel_second_power} always contains the vertex $(w_*,w_*)$ and also any independent set containing the vertex $(w_*,w_*)$ cannot contain vertices from the set $T_2$. Thus in the profile if we choose the vertex $(w_*,w_*)$, then maximal profile would be $(1, 0 , t(2t+1))$.
    \item We cannot choose more than $2t$ independent vertices from set $T_2$, since the set $T_2$ induces two disjoint copies of the cycle $C_{2t+1}$ in $W_{2t+1}^{\Box 2}$ and $\alpha(C_{2t+1}) = t$.
\end{enumerate}

Hence, the way to find candidates for maximal profiles is to enumerate over all choices of $\{1,2,\cdots, 2t\}$ for the independent set in $T_2$ and find the maximum independent set in $T_3$ when restricted to each choice. To do this, we set up an LP which is similar in spirit to Program~\ref{prog:alpha_QCIP}.

Let $B(T_i)$ be the incidence matrix of the graph induced by $T_i$ in $W_{2t+1}^{\Box 2}$ and let $B$ be the incidence matrix of $W_{2t+1}^{\Box 2}$. Let $x,y$ be indicator vectors for independent sets on $T_2$ and $T_3$, respectively. 
For a fixed choice of $k$ in $\{1,2,\cdots, 2t\}$ consider the LP:

\begin{align} \label{prog:profile_calc}
   r =  \max \quad & \mathbf{1}^T y \nonumber\\
   \text{subject to} \quad & B(T_2)^Tx \leq \mathbf{1} \nonumber\\
            &  B(T_3)^Ty\leq \mathbf{1}  \nonumber\\
            &  B^T \begin{pmatrix} 0 \\
            x \\
            y\end{pmatrix}\leq \mathbf{1} \nonumber \\
            &  \mathbf{1}^T x = k.  \nonumber 
\end{align}


Note that this LP returns the maximum $r$ such that there is an independent set 
with profile $(0,k,r)$. Performing this computation for $W_7^{\Box 2}$ we get these as our maximum profiles:
\[(1, 0, 21), (0, 1, 18), (0, 2, 18), (0, 3, 15), (0, 4, 13), (0, 5, 11), (0, 6, 10)\]
Notice that the profile $(0,1,18)$ is redundant. 
Following the notation in \cite{hahn1995ultimate}
we let
$a$ be the value on $T_1$, $b$ the value on $T_2$, and $c$ the value on $T_3$ in an optimal solution of Program~\ref{prog:1/chi_f}.
Now setting up the reduced LP according to these profiles, and these variables:
\begin{align} \label{prog:reduced_LP}
    \dfrac{1}{\chi_f(W_7^{\Box 2})} =  \min \quad & z \\
            \text{subject to} \quad & a + 21 c \leq z \nonumber\\
            &  2b + 18c \leq z  \nonumber \\
            &  3b + 15c \leq z   \nonumber\\
            &  4b + 13c \leq z   \nonumber\\
            &  5b + 11c \leq z   \nonumber\\
            &  6b + 10c \leq z   \nonumber\\
            &  a + 14b + 49c = 1.  \nonumber
\end{align}
Solving the above LP, we get $\chi_f(W_7^{\Box 2}) = \frac{39}{11}$.

We do a similar computation for $W_9^{\Box 2}$ and obtain the profiles 
\[(1, 0, 36), (0,1,32),(0, 2, 32), (0, 3, 28), (0, 4, 25), (0, 5, 22), (0, 6, 20), (0, 7, 18), (0 ,8, 17).\]
Again, notice that $(0,1,32)$ is redundant. Setting up and solving reduced LP gives us $\chi_f(W_9^{\Box 2}) = \frac{127}{37}$. We do similarly for all higher wheels as listed in the Table \ref{table:smallwheel}. 

In our computations, we observe that the following pattern holds: 
$(0,2,2t^2) , (0,2t,t^2+1)$ are maximal profiles and
the constraints corresponding to the profiles $(1,0,t(2t+1)) , (0,2,2t^2) , (0,2t,t^2+1)$ are always equalities. 
If this pattern holds for all $t$, then together with the first constraint in Program \ref{prog:1/chi_f} we obtain the
following system of linear equations
\begin{align*}
    a + t(2t+1)c &= z \\
    2b + 2t^2c &= z \\
    (2t)b +  (t^2+1)c &= z \\
    a + (4t+2)b + (2t+1)^2c &= 1. 
\end{align*}
Solving these gives $z = \frac{2t^2+t+1}{6t^2+7t+3}$.
Hence, if this pattern continues for large $t$, then this would imply $\chi_f(W^{\Box 2}_{2t+1}) = \frac{6t^2+7t+3}{2t^2+t+1}$ for all $t$.

\section{On the 5-Wheel}

In this section, we first outline a proof of Theorem \ref{thm:5_wheel_thirdpower_triangle}.
Building on this work, we go on to prove bounds for $\alpha(W^{\Box 4}_5)$ and $\alpha(W_{5}^{\Box 4}\Box K_3)$,  leading to a proof of Theorem \ref{thm:5_wheel_best_bound}. 

\subsection{On $\alpha(W_{5}^{\Box 3}\Box K_3)$}

We begin with some easy-to-verify computational facts
in Lemma~\ref{lemma:5_wheel_small_indendence_no}.
Any standard computational tool for calculating the independence number, for example SageMath, is sufficient to verify this lemma. As a result, we leave the verification of this lemma to the reader.

\begin{lemma}\label{lemma:5_wheel_small_indendence_no}
The following are known.
    \begin{enumerate}[$(i)$]
        \item $\alpha(W_{5}^{\Box 2}) = 11$.
        \item $\alpha(W_{5}^{\Box 2}\Box K_3)=29$.
        \item $\alpha(W_{5}^{\Box 3}) = 58$.
    \end{enumerate}
\end{lemma}

The following lemmas are less straightforward.
Also, the resulting LPs require a stronger solver to finish quickly. All of the assertions in the following lemma can be verified in a reasonable time using our implementations.

Let $I$ be an independent set in $W_{5}^{\Box 3}$. For $i\in \{*\}\cup \mathbb{Z}_5$, define 
\[ I_i = I\cap V(H_i^{(3)}) = \{(a_1, a_2, w_i) \in I: a_1, a_2 \in V(W_5)\}.\]
Define $\mathbf{i} = (|I_*|, |I_0|, |I_1|, |I_2|, |I_3|, |I_4|)$. 

\begin{figure}[!h]
    \centering
    \includegraphics[scale = 0.75]{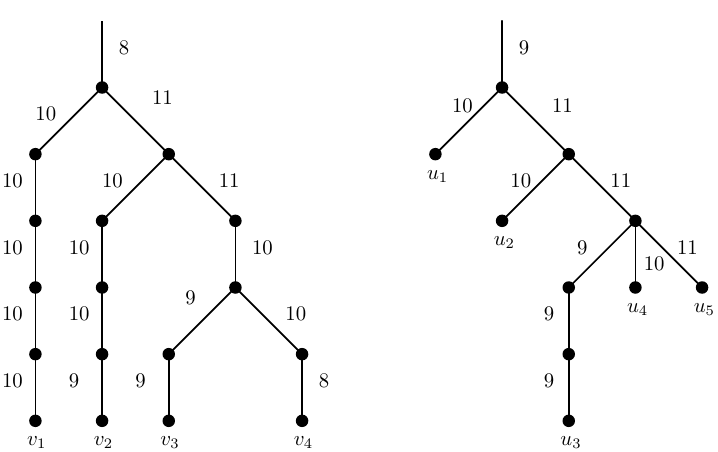}
    \caption{The branch and bound tree for assertion $(v)$ in Lemma~\ref{lemma:third_power_center_independence}.}
    \label{fig:BranchBound}
\end{figure}

\begin{figure}[!h]
    \centering
    \includegraphics[scale = 0.75]{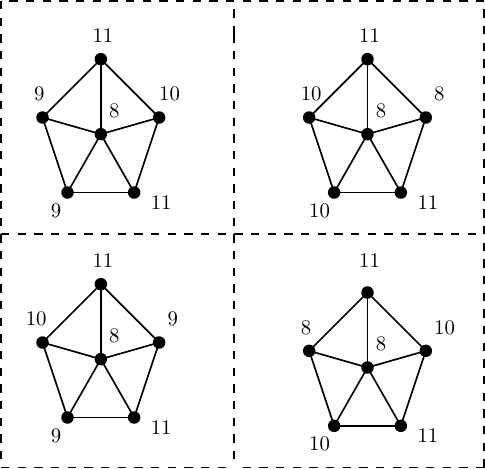}
    \caption{The possible cases when showing solutions $v_3,v_4$ solutions are infeasible in 
    assertion $(v)$ of Lemma~\ref{lemma:third_power_center_independence}.}
    \label{fig:62case}
\end{figure}

\begin{lemma}\label{lemma:third_power_center_independence}
Let $I$ be an independent set in $W_{5}^{\Box 3}$. 
    \begin{enumerate}[$(i)$]
        \item If $|I_*| = 11$, then $|I| \le 54$. 
        \item If $|I_*| = 10$, then $|I| \le 55$. 
        \item If $|I_*| = 9$ and if $|I_i| = 10$ for some $0\le i\le 4$, then $|I| \le 57$. 
        \item If $|I_*| \le 7$, then $|I| \le 57$. 
        \item If $|I| = 58$, then, without loss of generality, we have
        \[  \mathbf{i} \in \{(9,11, 9, 11, 9, 9), (8,11, 9, 10, 10, 10)\}.\]
    \end{enumerate}
\end{lemma}

\begin{proof}
    Assertion $(i)$ can be verified \href{https://github.com/Shivaramkratos/Ultimate_Independence_ratio_Python_Gurobi_code/blob/Lemma_6.2/6_2_i.py}{here},
$(ii)$ can be verified \href{https://github.com/Shivaramkratos/Ultimate_Independence_ratio_Python_Gurobi_code/blob/Lemma_6.2/6_2_ii.py}{here},
$(iii)$ can be verified \href{https://github.com/Shivaramkratos/Ultimate_Independence_ratio_Python_Gurobi_code/blob/Lemma_6.2/6_2_iii.py}{here},
and $(iv)$ can be verified \href{https://github.com/Shivaramkratos/Ultimate_Independence_ratio_Python_Gurobi_code/blob/Lemma_6.2/6_2_iv.py}{here}. 
Meanwhile, assertion $(v)$ requires a combinatorial justification, supported by multiple computations, to be verified.
We provide this proof now. 

Let $I$ be such that $|I| = 58$. Using $(i)$ and $(ii)$, we see that $|I_*|\le 9$, and using $(iv)$, we have $|I_*|\ge 8$. 
Let $\{J_0,\dots, J_4\} = \{I_0,\dots, I_4\}$ be a relabelling of the remaining sets $I_i$ such that
$|J_0|\geq |J_1| \geq |J_2|\geq |J_3| \geq |J_4|$.
Proceeding by a branch and bounding strategy, we
determine the resulting sizes of the sets $J_0,\dots, J_4$ in order,
observing that at all times $|J_0| \geq \frac{1}{5}(58 - |I_*|)$ and for all $i\geq 1$
$|J_i| \geq \frac{1}{5-i}(58 - |I_*|-\sum_{j=0}^{i-1}|J_j|)$.
The resulting decision tree is given in Figure~\ref{fig:BranchBound},
where edge labels determine the size of $J_i$
in order by index, beginning with $J_0$,
where the first edge of the tree determines if $|I_*| = 8$ or $9$.
Each leaf of the tree is
earliest point where we can decide the resulting set $\mathbf{i}$ is feasible or infeasible.

We claim that only leaves $v_2,u_3$ represent feasible solutions.
Thus, we are required to demonstrate that all other leaves represent infeasible solutions.
We begin by noting that assertion $(iii)$
implies no solution where there exists a $J_i$ with $|J_i| = 10$ is feasible.
Hence, solutions $u_1,u_2,$ and $u_4$ are infeasible.
Observe that by Lemma~\ref{lemma:ind_set_higher_power_large_center} and Proposition~\ref{prop:higher_wheel_second_power}
all distinct independent sets in $W^{\Box 2}_5$ of size $11$ intersect.
Thus, for all $i$, with addition in the index mod $5$,
no sets $I_i$ and $I_{i+1}$ can both satisfy $|I_i| = |I_{i+1}|= 11$.
It follows that at most $2$ sets $I_i$ have size $11$,
implying solution $u_5$ is infeasible.

It remains to be shown that solutions $v_1,v_3,$ and $v_4$ are infeasible.
This requires some more casework when determining how the relabelling
$J_0,J_1,J_2,J_3,J_4$ compares to the original labelling $I_0,I_1,I_2,I_3,I_4$,
in addition to further computer assistance.
To start, notice that the solution $v_1$ does not require casework, 
since all sets $J_i$ have size $10$.
The infeasibility of $v_1$ can be verified  
\href{https://github.com/Shivaramkratos/Ultimate_Independence_ratio_Python_Gurobi_code/blob/Lemma_6.2/6_2_v_v1.py}{here}.

Consider solution $v_3$ and $v_4$
sets $J_0$ and $J_1$ are size $11$, so without loss of generality
Lemma~\ref{lemma:ind_set_higher_power_large_center} and Proposition~\ref{prop:higher_wheel_second_power}
imply $J_0 = I_0$ and $J_1 = I_2$.
Then up to symmetries of $W_5$, $J_2 = I_1$ or $J_2 = I_4$.
This leads to two cases for $v_3$, see the left of Figure~\ref{fig:62case},
and two cases for $v_4$, see the right of Figure~\ref{fig:62case}.
The infeasibility of $v_3$ when $J_2 = I_1$ (see the top left)
can be verified  \href{https://github.com/Shivaramkratos/Ultimate_Independence_ratio_Python_Gurobi_code/blob/Lemma_6.2/6_2_v_v3_1.py}{here},
while the 
infeasibility of $v_3$ when $J_2 = I_4$ (see the bottom left)
can be verified  \href{https://github.com/Shivaramkratos/Ultimate_Independence_ratio_Python_Gurobi_code/blob/Lemma_6.2/6_2_v_v3_2.py}{here},
The infeasibility of $v_4$ when $J_4 = I_1$ (see the top right)
can be verified  \href{https://github.com/Shivaramkratos/Ultimate_Independence_ratio_Python_Gurobi_code/blob/Lemma_6.2/6_2_v_v4_1.py}{here},
while the 
infeasibility of $v_4$ when $J_2 = I_4$ (see the bottom right)
can be verified \href{https://github.com/Shivaramkratos/Ultimate_Independence_ratio_Python_Gurobi_code/blob/Lemma_6.2/6_2_v_v4_2.py}{here}.  This completes the proof.
\end{proof}

\begin{figure}
    \centering
    \includegraphics[scale = 0.75]{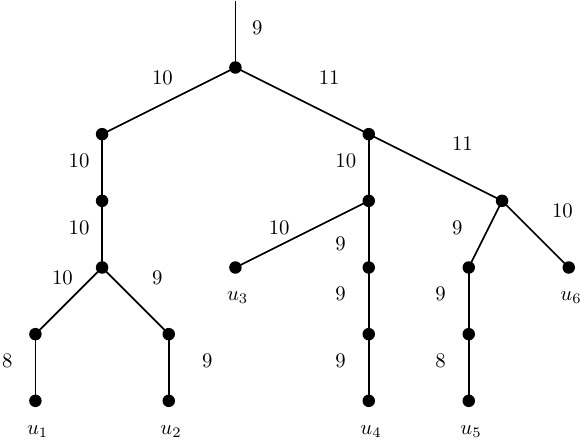}
    \caption{The branch and bound tree for independent sets of size $57$ in $W^{\Box 3}_5$
    where $|I_*| = 9$, as mentioned in
    in Lemma~\ref{lemma:57 center 9}.}
    \label{fig:BB}
\end{figure}

\begin{lemma}\label{lemma:57 center 9}
   If $I$ is an independent set in  
   $W_{5}^{\Box 3}$, such that $|I| = 57$ and $|I_*| = 9$, then, without loss of generality, we have  
   \[
    \mathbf{i} \in \{(9, 11, 8, 11, 9, 9), (9, 11, 9, 11, 9, 8), (9, 11, 9, 10, 9, 9), (9, 9, 10, 9, 10, 10)\}.
   \]
\end{lemma}

\begin{proof} Using the branch-and-bound technique as described in the proof of Lemma \ref{lemma:third_power_center_independence}, we can construct the decision tree for $I$ as shown in Figure \ref{fig:BB}.  

Since $\alpha(W_5^{\Box 2}\Box K_3) = 29$, we immediately see that any feasible solution with $|I_*| = 9$ cannot have sets $I_i,I_{i+1}$ (addition mod $5$)
such that $|I_i| + |I_{i+1}|>20$.
This shows that the solutions $u_3$ and $u_6$ are infeasible. 

The solution $u_1$ is shown to be infeasible \href{https://github.com/Shivaramkratos/Ultimate_Independence_ratio_Python_Gurobi_code/blob/Lemma_6.3/6_3_u1.py}{here}. In the solution $u_2$, there are two possibilities: either $\mathbf{i} = (9, 9, 9, 10, 10, 10)$ or $\mathbf{i} = (9, 9, 10, 9, 10, 10)$. The former case is shown to be infeasible \href{https://github.com/Shivaramkratos/Ultimate_Independence_ratio_Python_Gurobi_code/blob/Lemma_6.3/6_3_u2_2.py}{here}, the latter case is possible. The cases $u_4$ and $u_5$ are feasible, but we note that for each of $u_4$ and $u_5$, there is a unique way to arrange set sizes, up to rotation symmetry, such that 
there are no sets $I_i,I_{i+1}$ (addition mod $5$)
such that $|I_i| + |I_{i+1}|>20$.
This completes the proof.
\end{proof}

Let $V(K_3) = \{w_i: i\in \mathbb{Z}_3\}$ and $S$ denote an independent set in $W_{5}^{\Box 3}\Box K_3$. For $i\in \mathbb{Z}_3$ and $j\in \{*\}\cup \mathbb{Z}_5$, let
\[S_i = \{(a_1, a_2, a_3, w_i) \in S: a_1, a_2, a_3\in V(W_5)\} \text{ and } S_{i, j} =  \{(a_1, a_2, w_j, w_i) \in S: a_1, a_2\in V(W_5)\}.\] Without loss of generality, let $|S_0|\ge |S_1|\ge |S_2|$. Denote by $\mathbf{s} = (|S_0|, |S_1|, |S_2|)$. 

\begin{lemma}\label{lemma:thirdpower_wit5h_K3_infeasible_sets}
Let $S$ be an independent set in $W_{5}^{\Box 3}\Box K_3$. Then the following hold.
\begin{enumerate}[$(i)$]
    \item There is no $S$ with $\mathbf{s} \in \{(58, 57, 55), (58, 56, 56)\}$.
    \item There is no $S$ with $\mathbf{s} = (57, 57, 57)$. 
    \item There is no $S$ with $\mathbf{s} = (58, 58, 54)$ and $(|S_{0, *}|, |S_{1, *}|, |S_{2, *}|) \in \{(9,9, 11), (9,8,10) ,(9,9,9)\}$. 
\end{enumerate}
\end{lemma}

\begin{proof} In this proof, we repeatedly use Lemma \ref{lemma:third_power_center_independence} $(v)$ without mention.

The infeasibility of the case $\mathbf{s} = (58, 56, 56)$ and $(|S_{0,*}|,|S_{0,0}|,\dots,|S_{0,4}|) = (9,11, 9, 11, 9, 9)$, 
    can be verified {\href{https://github.com/Shivaramkratos/Ultimate_Independence_ratio_Python_Gurobi_code/blob/Lemma_6.4/W5_K3_170_58-56-56-1.py}{here}}.
    The infeasibility of the case $\mathbf{s} = (58, 56, 56)$ and $(|S_{0,*}|,|S_{0,0}|,\dots,|S_{0,4}|) = (8,11, 9, 10, 10, 10)$, 
    can be verified {\href{https://github.com/Shivaramkratos/Ultimate_Independence_ratio_Python_Gurobi_code/blob/Lemma_6.4/W5_K3_170_58-56-56-2.py}{here}}.

    The infeasibility of the case $\mathbf{s} = (58, 57, 55)$ and $(|S_{0,*}|,|S_{0,0}|,\dots,|S_{0,4}|) = (9,11, 9, 11, 9, 9)$, 
    can be verified \href{https://github.com/Shivaramkratos/Ultimate_Independence_ratio_Python_Gurobi_code/blob/Lemma_6.4/W5_K3_170_58-57-55-1.py}{here}.
    The infeasibility of the case $\mathbf{s} = (58, 57, 55)$ and $(|S_{0,*}|,|S_{0,0}|,\dots,|S_{0,4}|) = (8,11, 9, 10, 10, 10)$, 
    can be verified {\href{https://github.com/Shivaramkratos/Ultimate_Independence_ratio_Python_Gurobi_code/blob/Lemma_6.4/W5_K3_170_58-57-55-2.py}{here}}.
    This proves $(i)$.

Suppose $\mathbf{s} = (57, 57, 57)$. Using Lemma \ref{lemma:third_power_center_independence}, we see that $|S_{i, *}|\le 9$ for all $i\in \{0,1,2\}$. Note that
\begin{align*}
        171 = |S| & = (|S_{0, *}| + |S_{1, *}| + |S_{2, *}|) + \sum_{j = 0}^4 \left(|S_{0, j}| + |S_{1, j}| + |S_{2, j}|\right) \\
        & \le (|S_{0, *}| + |S_{1, *}| + |S_{2, *}|) + 5 \cdot 29.
\end{align*}
This implies $(|S_{0, *}| + |S_{1, *}| + |S_{2, *}|)\ge 26$. Assume without loss of generality that $|S_{0, *}| \ge |S_{1, *}|\ge |S_{2, *}|$. Then 
\[(|S_{0, *}|,|S_{1, *}|, |S_{2, *}|) \in \{(9,9,9), (9,9,8)\}.\]
Without loss of generality $|S_{0, *}| = 9$.
By Lemma~\ref{lemma:57 center 9} we have 
\[(|S_{0,*}|,|S_{0,0}|,\dots,|S_{0,4}|) \in \{(9, 11, 8, 11, 9, 9), (9, 11, 9, 11, 9, 8), (9, 11, 9, 10, 9, 9), (9, 9, 10, 9, 10, 10)\},\]
First, letting $(|S_{0, *}|, |S_{1, *}|, |S_{2, *}|) = (9,9,9)$, then we verify each case of $(|S_{0,*}|,|S_{0,0}|,\dots,|S_{0,4}|)$
\href{https://github.com/Shivaramkratos/Ultimate_Independence_ratio_Python_Gurobi_code/blob/Lemma_6.4/W5_K3_171_57-57-57-9-9-9-1.py}{here},
\href{https://github.com/Shivaramkratos/Ultimate_Independence_ratio_Python_Gurobi_code/blob/Lemma_6.4/W5_K3_171_57-57-57-9-9-9-2.py}{here},
\href{https://github.com/Shivaramkratos/Ultimate_Independence_ratio_Python_Gurobi_code/blob/Lemma_6.4/W5_K3_171_57-57-57-9-9-9-3.py}{here}, and
\href{https://github.com/Shivaramkratos/Ultimate_Independence_ratio_Python_Gurobi_code/blob/Lemma_6.4/W5_K3_171_57-57-57-9-9-9-4.py}{here}, respectively. 
Second, we let $(|S_{0, *}|, |S_{1, *}|, |S_{2, *}|) = (9,9,8)$
then we verify each case of $(|S_{0,*}|,|S_{0,0}|,\dots,|S_{0,4}|)$
\href{https://github.com/Shivaramkratos/Ultimate_Independence_ratio_Python_Gurobi_code/blob/Lemma_6.4/W5_K3_171_57_57_57_9_9_8_1.py}{here},
\href{https://github.com/Shivaramkratos/Ultimate_Independence_ratio_Python_Gurobi_code/blob/Lemma_6.4/W5_K3_171_57_57_57_9_9_8_2.py}{here},
\href{https://github.com/Shivaramkratos/Ultimate_Independence_ratio_Python_Gurobi_code/blob/Lemma_6.4/W5_K3_171_57_57_57_9_9_8_3.py}{here},
and \href{https://github.com/Shivaramkratos/Ultimate_Independence_ratio_Python_Gurobi_code/blob/Lemma_6.4/W5_K3_171_57_57_57_9_9_8_4.py}{here}, respectively. 
This proves $(ii)$.

Now, let $\mathbf{s} = (58, 58, 54)$. First, let $(|S_{0, *}|, |S_{1, *}|, |S_{2, *}|) = (9,9,11)$. Then, without loss of generality, we can assume that \[(|S_{0,*}|,|S_{0,0}|,\dots,|S_{0,4}|) = (9,11, 9, 11, 9, 9) \text{ and }(|S_{1,*}|,|S_{1,0}|,\dots,|S_{1,4}|) = (9,9, 11, 9,11, 9).\]
The infeasibility of this case is verified \href{https://github.com/Shivaramkratos/Ultimate_Independence_ratio_Python_Gurobi_code/blob/Lemma_6.4/W5_K3_170_58_58_54_9_9_11.py}{here}.

Second, let $(|S_{0, *}|, |S_{1, *}|, |S_{2, *}|) = (9,8,10)$.  Then, we just assume that \[(|S_{1,*}|,|S_{1,0}|,\dots,|S_{1,4}|) = (8,11, 9, 10, 10, 10).\] The infeasibility of this case is verified \href{https://github.com/Shivaramkratos/Ultimate_Independence_ratio_Python_Gurobi_code/blob/Lemma_6.4/W5_K3_170_58_58_54_9_8_10.py}{here}.

Thirdly, let $(|S_{0, *}|, |S_{1, *}|, |S_{2, *}|) = (9,9,9)$. without loss of generality, we can assume that \[(|S_{0,*}|,|S_{0,0}|,\dots,|S_{0,4}|) = (9,11, 9, 11, 9, 9) \text{ and }(|S_{1,*}|,|S_{1,0}|,\dots,|S_{1,4}|) = (9,9, 11, 9,11, 9).\]
The infeasibility of this case is verified \href{https://github.com/Shivaramkratos/Ultimate_Independence_ratio_Python_Gurobi_code/blob/Lemma_6.4/W5_K3_170_58_58_54_9_9_9.py}{here}. This concludes the proof of $(iii)$.
\end{proof}

\begin{lemma}\label{lemma:large_sets_third_power_triangle}
Let $S$ be an independent set in $W_{5}^{\Box 3}\Box K_3$.
If $|S| = 170$, then 
\[\mathbf{s} \in \{(58,58,54), (57,57,56)\}.\]
\end{lemma}

\begin{proof}
By Lemma~\ref{lemma:5_wheel_small_indendence_no} $(iii)$, for all $i$, $|S_i| \leq 58$.
By our assumption $|S| = |S_0|+|S_1|+|S_2| = 170$.
Thus, it is easy to see, using the branch-and-bound technique used earlier, that 
\[\mathbf{s} \in \{(58, 58, 54), (58, 57, 55), (58, 56, 56), (57, 57, 56)\}.\]
Using Lemma~\ref{lemma:thirdpower_wit5h_K3_infeasible_sets} $(i)$ we see that $\mathbf{s} \notin \{(58, 57, 55), (58, 56, 56)\}$.
\end{proof} 

\begin{figure}[!h]
\scalebox{0.9875}{
\centering
\begin{subfigure}{0.3\textwidth}
    \includegraphics[scale = 0.33]{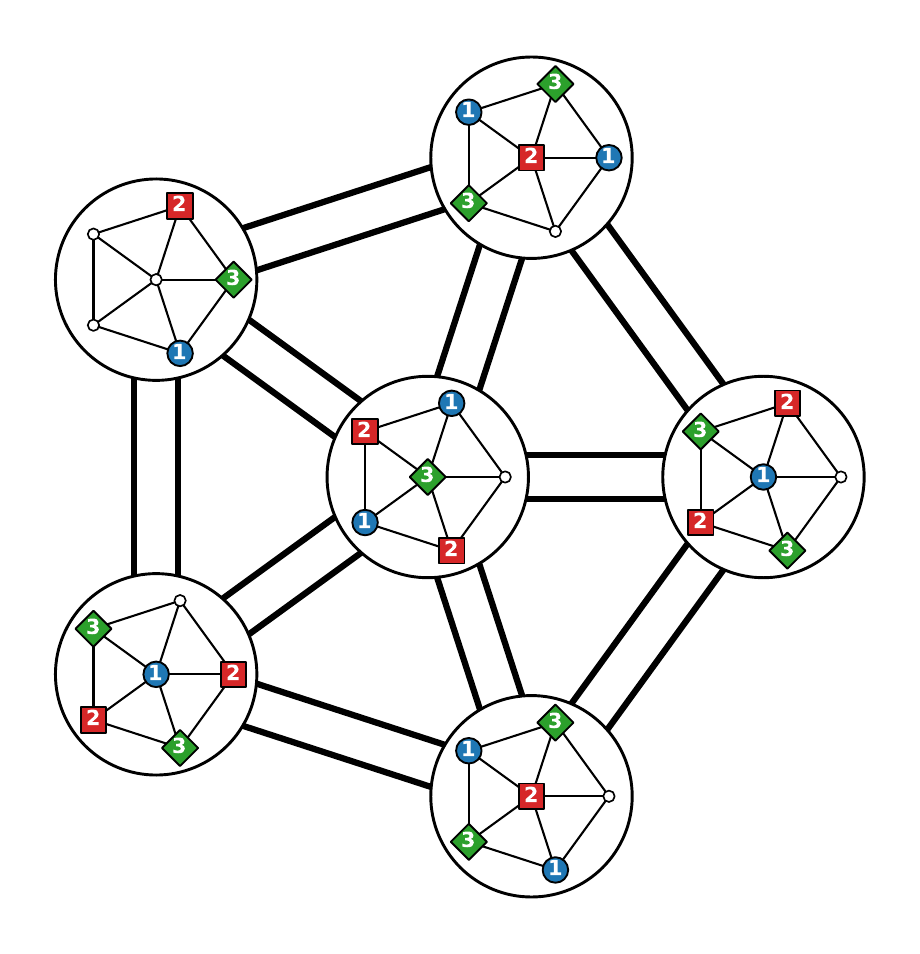}
    \caption{$I_*$.}
    \label{fig:star}
\end{subfigure}
\hfill
\begin{subfigure}{0.3\textwidth}
    \includegraphics[scale = 0.33]{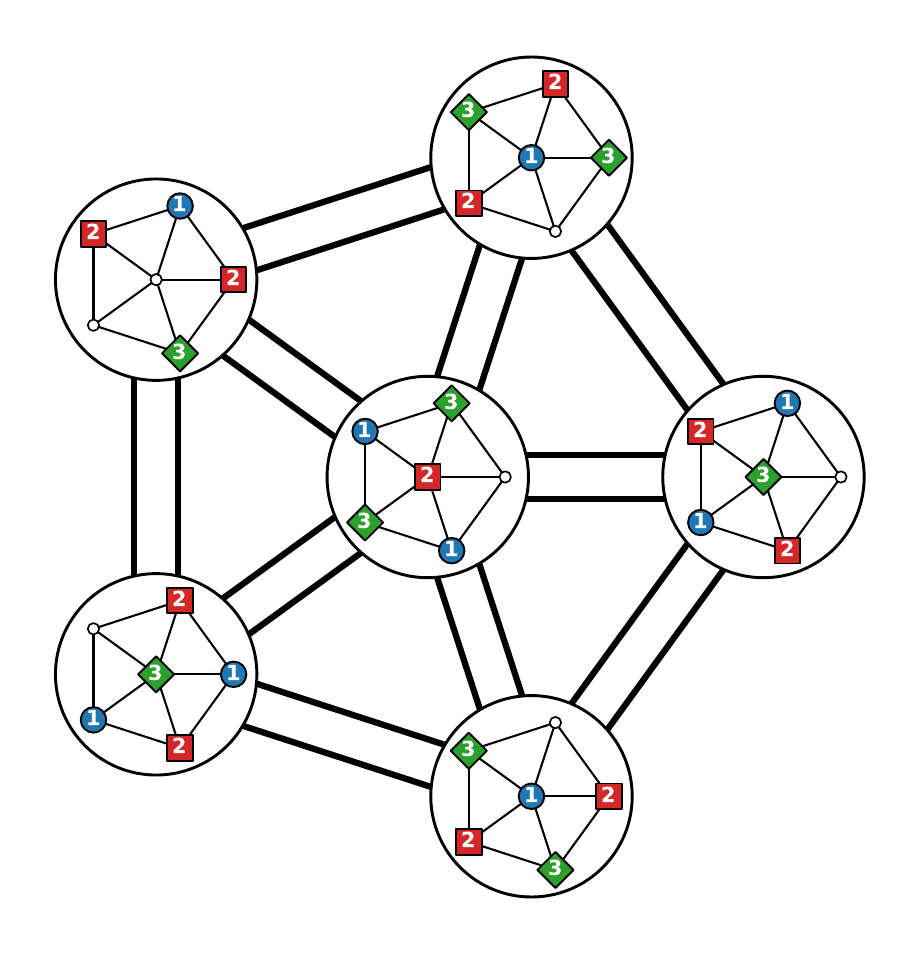}
    \caption{$I_0$.}
    \label{fig:1}
\end{subfigure}
\hfill
\begin{subfigure}{0.3\textwidth}
    \includegraphics[scale = 0.33]{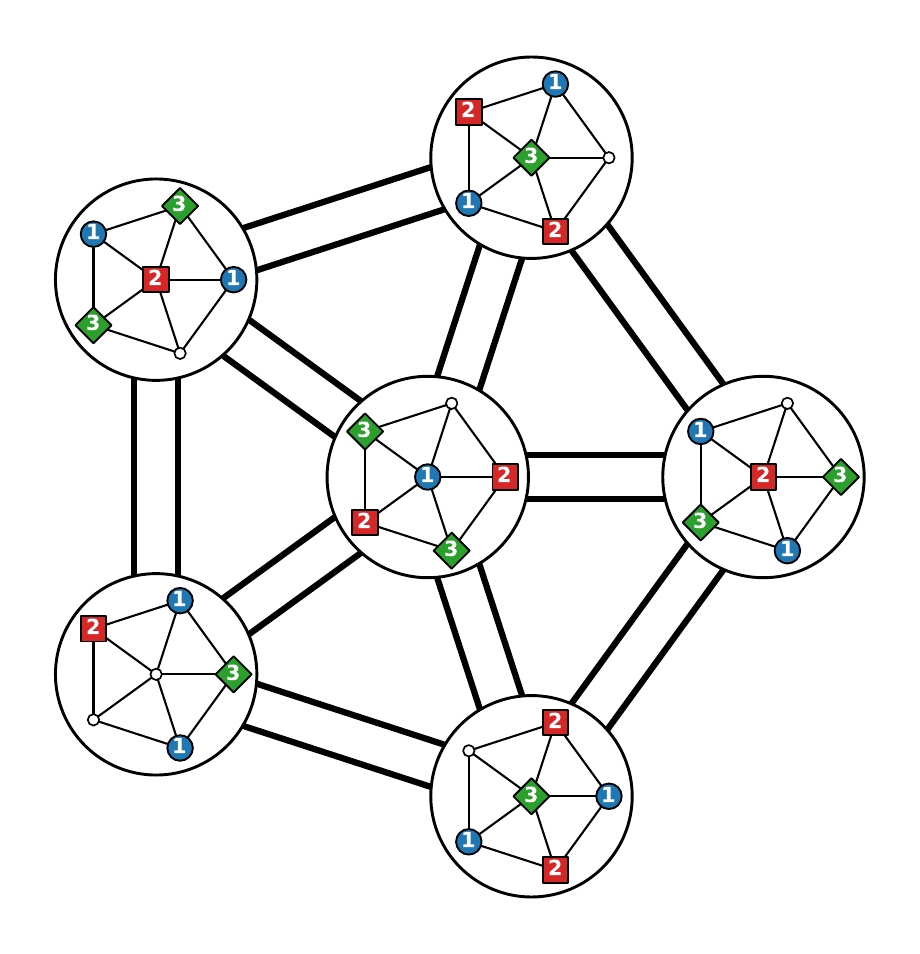}
    \caption{$I_1$.}
    \label{fig:2}
\end{subfigure}
}
\scalebox{0.9875}{
\hfill
\begin{subfigure}{0.3\textwidth}
    \includegraphics[scale = 0.33]{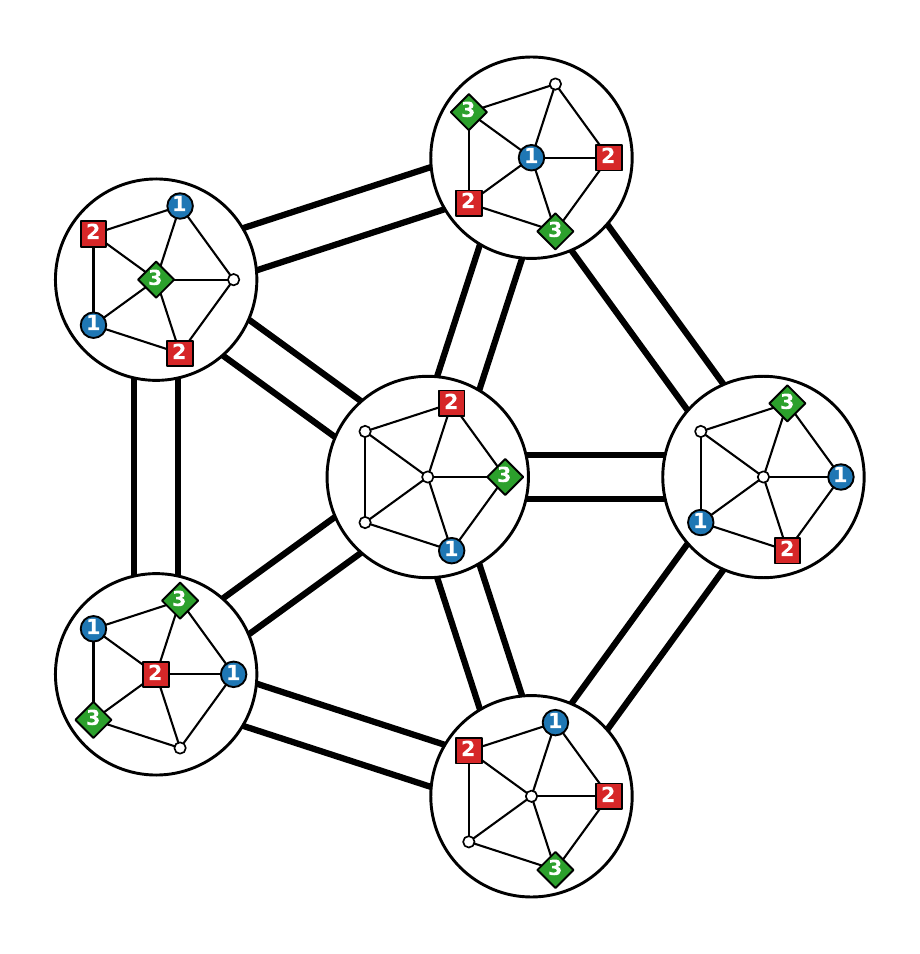}
    \caption{$I_2$.}
    \label{fig:3}
\end{subfigure}
\hfill
\begin{subfigure}{0.3\textwidth}
    \includegraphics[scale = 0.33]{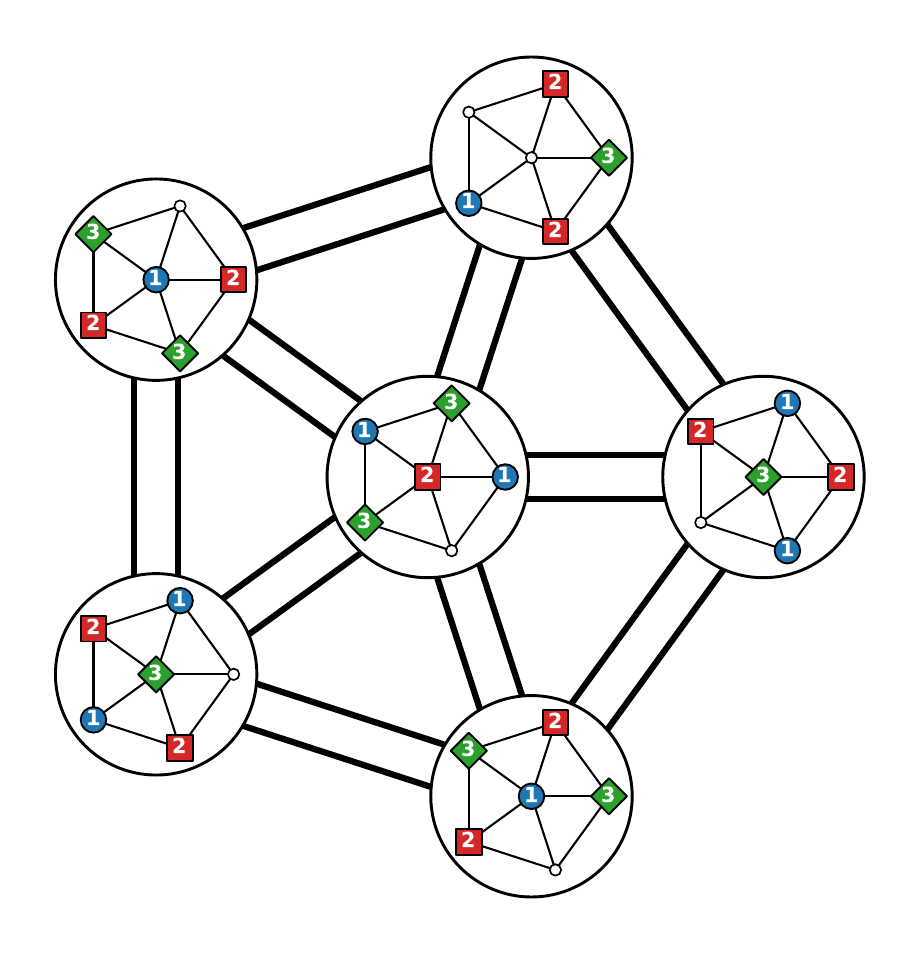}
    \caption{$I_3$.}
    \label{fig:4}
\end{subfigure}
\hfill
\begin{subfigure}{0.3\textwidth}
    \includegraphics[scale = 0.33]{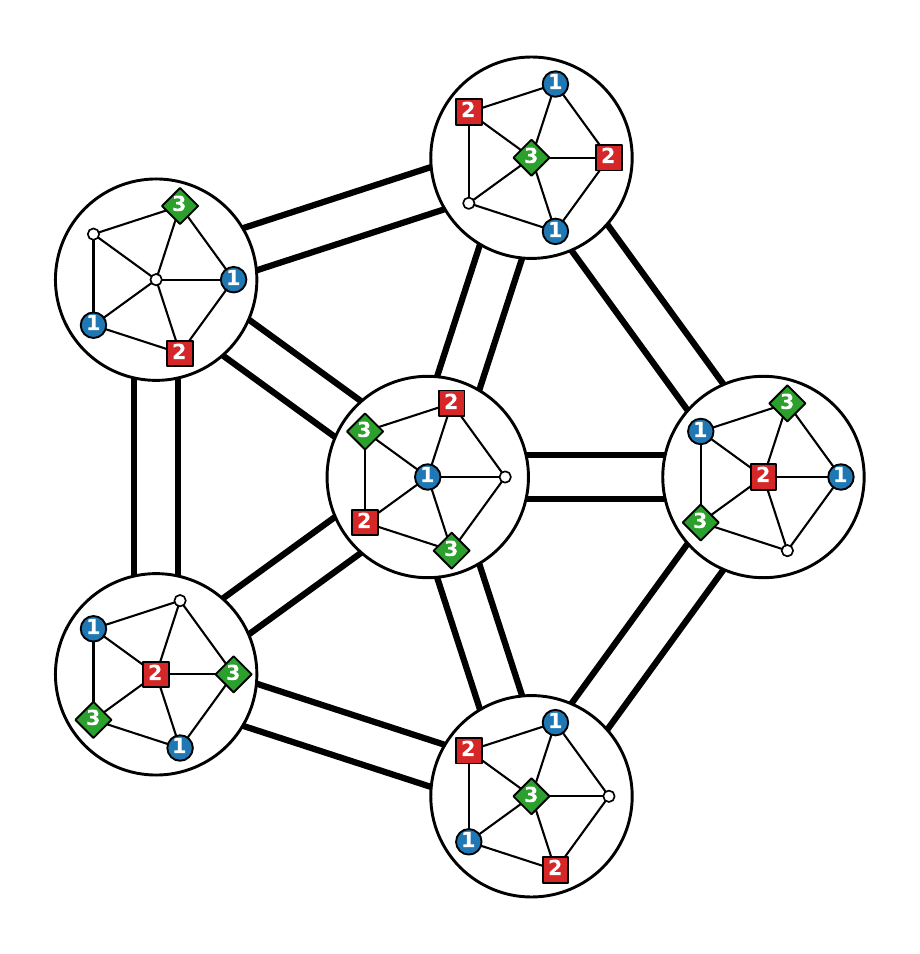}
    \caption{$I_4$.}
    \label{fig:5}
\end{subfigure}
}
\caption{An independent $S$ set of size $170$ in $W^{\Box 3}_5 \Box K_3$.
Each set $S_i$ is represented by a colour in $W^{\Box 3}_5$, while $W^{\Box 3}_5$ is displayed using $6$ copies $W^{\Box 2}_5$
via the partition $I_{*},I_0,\dots, I_4$.
}
\label{fig:independent_set_170}
\end{figure}

We are now prepared to finish the proof of Theorem \ref{thm:5_wheel_thirdpower_triangle}. 

\begin{proof}[Proof of Theorem \ref{thm:5_wheel_thirdpower_triangle}]
We begin by noting that Figure~\ref{fig:independent_set_170} displays an independent set of size $170$ in $W^{\Box 3}_5 \Box K_3$.
Hence, $\alpha(W^{\Box 3}_5 \Box K_3) \geq 170$ as desired.
Now, suppose to the contrary that $|S|\ge 171$. We can assume that $|S|=171$, by deleting some vertices from $S$ if needed. Since $|S_i|\le 58$ by Lemma \ref{lemma:5_wheel_small_indendence_no}, and $|S_0|+|S_1|+|S_2| = 171$ by assumption, it is easy to see that
\[ \mathbf{x} \in \{(58, 58, 55), (58, 57, 56), (57, 57, 57)\}.\]
By Lemma \ref{lemma:thirdpower_wit5h_K3_infeasible_sets}, we see that such an $S$ does not exist. This completes the proof.
\end{proof}

\subsection{On $\alpha(W_5^{\Box 4})$}

Using Theorem \ref{thm:5_wheel_thirdpower_triangle} and Lemma \ref{lemma:5_wheel_small_indendence_no}, we immediately see that 
\begin{align*}
    \alpha(W_5^{\Box 4}) &\le \alpha(W_5^{\Box 3}\Box K_3) + \alpha(W_5^{\Box 3}\Box P_3)\\
    & \le 170 + 3\cdot \alpha(W_5^{\Box 3})\\
    & = 170 + 3\cdot 58 = 344.
\end{align*}

We can improve this a bit further. 

\begin{theorem}\label{thm:5_wheel_fourth_power} We have 
\[ \alpha(W^{\Box 4}_5) \le 343.\]
\end{theorem}

\begin{proof}
Suppose, to the contrary, that there exists an independent set $I$ in $W^{\Box 4}$ such that $|I|=344$. For $i, j\in \{*, 0, \ldots, 4\}$, let
\[I_i = I\cap V(H_i^{(4)})\quad \text{ and }\quad I_{i, j} = I\cap V(H_{i,j}^{(4)}).\] Clearly, $I = \cup_{i}I_i$ and  $I_i = \cup_{j} I_{i,j}.$ Using Theorem \ref{thm:5_wheel_thirdpower_triangle}, for any $i\in \mathbb{Z}_5$, we have
\[|I_*| + |I_{i}| + |I_{i+1}| \le 170.\] 
Summing up over $i \in \mathbb{Z}_5$ gives
\[ 5|I_*| + 2 \sum_{i=0}^4 |I_i| \le 850.\]
Thus
\[ 3|I_*| \le 850 - 2|I| = 850 - 688 = 162, \]
implying $|I_*|\le 54$. We also see that 
\[54\ge |I_*|\ge |I| - \sum_{i=0}^4 |I_i| \ge 344 - 5\cdot 58 = 54,\] i.e., here 
all inequalities are equalities. In other words, $|I_*| = 54$ and $|I_i| = 58$ for all $i\in \mathbb{Z}_5$.
Also,
\begin{equation}\label{eq:343_one}
54 = |I_*| = |I_{*, *}| + \sum_{i = 0}^4 |I_{i,*}|.
\end{equation}
By Lemma \ref{lemma:third_power_center_independence} (or its contrapositive), $|I_{i,*}|\le 9$ since $|I_i| = 58$ for $i\in \{0, \ldots, 4\}$. By \eqref{eq:343_one}, $|I_{*,*}|\ge 9$. Let $A = \{i : |I_{i,*}| = 9, 0\le i\le 4\}$. Since $|I_{*,*}|\le 11$, we see that $|A|\ge 3$ by \eqref{eq:343_one}. Without loss of generality, we can assume that $\{0,1\}\subset A$.  

Note that $J = I_* \cup I_0\cup I_1$ is an independent set in $W_5^{\Box 3}\Box K_3$ with $(|I_{*, *}|, |I_{0, *}|, |I_{1, *}|) = (|I_{*,*}|, 9, 9)$. If $|I_{*, *}|\in \{9, 11\}$, then such a $J$ does not exist by Lemma \ref{lemma:thirdpower_wit5h_K3_infeasible_sets} $(iii)$. So the only possibility is $|I_{*, *}| = 10$. By \eqref{eq:343_one}, we see that $|A| = 4$. Without loss of generality, we can assume that $|I_{4, *}| = 8$. Then $J' = I_{*}\cup I_0 \cup I_4$ is an independent set in $W_5^{\Box 3}\Box K_3$ with $(|I_{*, *}|, |I_{0, *}|, |I_{4, *}|) = (10, 9, 8)$. By Lemma \ref{lemma:thirdpower_wit5h_K3_infeasible_sets} $(iii)$, such a set $J'$ does not exist. The proof is complete. 
\end{proof}

\subsection{On $\alpha(W_{5}^{\Box 4}\Box K_3)$}

\begin{theorem}\label{thm:5_wheel_fourthpower_triangle}
    $\alpha(W_5^{\Box 4}\Box K_3) \le 1019$.
\end{theorem}

\begin{proof} 
Assume $V(K_3) = \{w_i: i\in \mathbb{Z}_3\}$ and let $S$ be an independent set in $W_5^{\Box 4}\Box K_3$. For $i\in \mathbb{Z}_3$ and $j \in \{*\}\cup \mathbb{Z}_5$, define
\[ S_i = \{(a_1, \ldots, a_4, w_i): a_1,\ldots, a_4\in V(W_5)\}\]
and 
\[S_{i,j} = \{(a_1,a_2, a_3, w_j, w_i): a_1, a_2, a_3 \in V(W_5)\}.\]

Note here that $|S_{i,j}|\le \alpha({W_5^{\Box 3}}) = 58$ using Lemma \ref{lemma:5_wheel_small_indendence_no}, and $\sum_{i=0}^2 
|S_{i,j}|\le \alpha(W_{5}^{\Box 3}\Box K_3) = 170$ by Theorem \ref{thm:5_wheel_thirdpower_triangle}.

Assume to the contrary that $|S| = 1020$. Then, for all $j\in \{*\}\cup \mathbb{Z}_5$, we have
\begin{equation} \label{eq:fourth_power_triangle_1}
 \sum_{i=0}^2 |S_{i,j}|=170.
\end{equation}
Without loss of generality, we can assume that $|S_{0,*}|\ge |S_{1,*}|\ge |S_{2,*}|$. By Lemma \ref{lemma:large_sets_third_power_triangle}, we see that $|S_{2,*}|\in \{54, 56\}$ and $|S_{0,*}|\ge |S_{1,*}|\ge 57$. For $i\in \{0,1\}$, define $A_i = \{0\le j\le 4: |S_{i,j}|\le 56\}$. By Lemma \ref{lemma:large_sets_third_power_triangle}, we have $|A_i|\ge 3$ for $i\in \{0,1\}$. Thus, there exists an index $j_0\in A_0\cap A_1$ by the Pigeonhole Principle. Thus no ordering of elements in $\{|S_{0,j_0}|, |S_{1,j_0}|,|S_{2,j_0}|\}$ lies in $\{(58, 58, 54), (57, 57, 56)\}$ and so 
\[|S_{0,j_0}| + |S_{1,j_0}| + |S_{2,j_0}|\le 169.\]
This contradicts \eqref{eq:fourth_power_triangle_1}.
\end{proof}

Theorem \ref{thm:5_wheel_best_bound} now follows using the Theorems \ref{thm:5_wheel_fourthpower_triangle} and \ref{thm:upper_bound_triangle}.

\section{Future Work}

Of course, the most natural problem for future work remains Conjecture~\ref{Conj: Odd Wheels}.
It seems to us that there remains significant room to theoretically, i.e., without computer assistance,
improve our bound in Theorem~\ref{thm:higher_wheel_best_bound}.
Recalling that Theorem~\ref{thm:higher_wheel_best_bound} is obtained by estimating $\alpha(W_{2t+1}^{\Box 2}\Box K_3)$
for all values of $t$, we believe that determining the true value of $\alpha(W_{2t+1}^{\Box 2}\Box K_3)$ is of interest.
Noting our computational work summarized in Table~\ref{table:smallwheel} and Appendix~A, we strongly believe the following.

\begin{conjecture}\label{Conj: Squares} For all $t\geq 2$,
    \[
    (2t+2)\alpha(W_{2t+1}) - \alpha(W_{2t+1}^{\Box 2}\Box K_3) = t-1.
    \] 
    Hence, the largest $3$-colourable subgraph of $W_{2t+1}^{\Box 2}$
    is order $4t^2+5t+3$.
\end{conjecture}

It is also interesting to ask if we can determine the true value of $\alpha(W^{\Box 4}_5)$.
Some heuristic computer searches suggest that $\alpha(W^{\Box 4}_5) \geq 338$.
We conjecture that this is the true value of $\alpha(W^{\Box 4}_5)$.
If correct, this would yield an improved bound $\mathscr{I}(W_5)$ compared to Theorem~\ref{thm:5_wheel_best_bound}.

\begin{conjecture}
    $\alpha(W^{\Box 4}_5)= 338$
\end{conjecture}

Our work computing $\chi_f(W_{2t+1}^{\Box 2})$ for small $t$ leads us to conjecture the following.
Notice that it is sufficient to prove our observations about the extremal points 
of the reduced LP for small $t$ generalize to all $t$. 

\begin{conjecture}\label{conj:chi_fSquare}
    $\chi_f(W_{2t+1}^{\Box 2}) = \frac{6t^2 + 7t + 3}{2t^2 + t +1}$
\end{conjecture}

By bootstrapping our methods in Section~5.2 for computing maximal profiles, we were also able to compute 
$\chi_f(W_5^{\Box 3}) = \frac{722}{189}$.
This was a dramatically more intensive computation than computing the fractional chromatic number of second powers.
Can one obtain an analogous conjecture to Conjecture~\ref{conj:chi_fSquare}
for third powers? If so, is there an obvious way to generalize these polynomials to guess the value of 
$\chi_f(W_{2t+1}^{\Box \ell})$ for any $t$ and $\ell$?

We are also curious when Theorem~\ref{thm:upper_bound_triangle} provides a better estimate on $\mathscr{I}(W_{2t+1})$,
for a fixed $\ell$, than $\frac{1}{\chi_f(W^{\Box \ell}_{2t+1})}$.
We have seen that this is not the case for $t=2$ and $\ell= 2,3$,
but for all other pairs $t$ and $\ell$
where we have both parameters Theorem~\ref{thm:upper_bound_triangle} provides a better estimate
than $\frac{1}{\chi_f(W^{\Box \ell}_{2t+1})}$.
We conjecture that this is always the case for sufficiently large $t$.
Zhu \cite[Theorem~2]{zhu1996bounds} showed that for all graphs $G$,
$\mathscr{I}(G) = \lim_{\ell\rightarrow \infty }\frac{1}{\chi_f(W^{\Box \ell}_{2t+1})}$.

\begin{conjecture}\label{conj: better estimate}
    For all $t\geq 3$ and all $\ell\geq 1$,
    \[
    \frac{\alpha(W_{2t+1}^{\Box \ell }\Box K_3)}{3(2t+2)^{\ell}} < \frac{1}{\chi_f(W^{\Box \ell}_{2t+1})}.
    \]
\end{conjecture}

Our final conjecture combines our intuition for Theorem~\ref{thm:upper_bound_triangle} and Conjecture~\ref{conj: better estimate}.

\begin{conjecture}\label{conj: final}
    There exists a constant $c>0$ such that if $G$ is a graph with $\chi(G) > \omega(G)$
    and
    $H$ is a largest induced subgraph of $G$ with $\chi(H) \leq \omega(G)$,
    then 
    $\chi_f(G) < c\Big(\frac{\omega(G)|V(G)|}{|V(H)|}\Big)$.
\end{conjecture}

\section*{Acknowledgements}

The authors would like to thank G\v{e}na Hahn
for introducing us to this problem through their
talk \textit{Resurrection – revisiting old problems}
at the Canadian Math Society winter meeting in 2024.
We would also like to thank Sean Kim for assisting us in our ultimately unsuccessful 
attempts to compute the fractional chromatic number of $W^{\Box 3}_7$ and $W^{\Box 3}_9$
on the FIR supercomputer at Simon Fraser University.
Clow is supported by the Natural Sciences and Engineering Research Council of Canada (NSERC) through PGS D-601066-2025.

\bibliographystyle{plain}
\bibliography{ref}

\vspace{0.4cm}
\noindent Alexander Clow, Email: {\tt alexander\_clow@sfu.ca}\\
\textsc{Department of Mathematics, Simon Fraser University, Burnaby, BC, Canada}\\[0.5pt]

\noindent Hitesh Kumar, Email: {\tt hitesh.kumar.math@gmail.com, hitesh\_kumar@sfu.ca}\\ 
\textsc{Department of Mathematics, Simon Fraser University, Burnaby, BC, Canada}\\[0.5pt]

\noindent Shivaramakrishna Pragada, Email: {\tt shivaramakrishna\_pragada@sfu.ca, shivaramkratos@gmail.com}\\ 
\textsc{Department of Mathematics, Simon Fraser University, Burnaby, BC, Canada}

\newpage

\section*{Appendix A}

In this appendix, we 
provide independent sets that demonstrate
$\alpha(W^{\Box 2}_{2t+1}\Box K_3) \geq 4t^2+5t+3$
for all $2 \leq t \leq 5$.
Although we have examples of independent sets of the prescribed size in $W^{\Box 2}_{13}\Box K_3$, we do not include them here, since the resulting figure is not easily readable.
Recalling that independent sets in $W^{\Box 2}_{2t+1}\Box K_3$ correspond to $3$-colourable subgraphs of $W^{\Box 2}_{2t+1}$
we present each independent set as a partial $3$-colouring of $W^{\Box 2}_{2t+1}$.

We now give a more formal description of 
our figures.
Supposing $S$ is an independent set in 
$W^{\Box 2}_{2t+1}\Box K_3$, we will encode $S$ in a drawing of $W^{\Box 2}_{2t+1}$ as follows:
if $(i,j,x) \in S$, then we will colour vertex $(i,j)$ by colour $x$.
If for all $x \in \{1,2,3\}$, $(i,j,x) \notin S$, then vertex $(i,j)$ does not receive a colour.

\begin{figure}[!h]
    \centering
    \includegraphics[scale = 0.65]{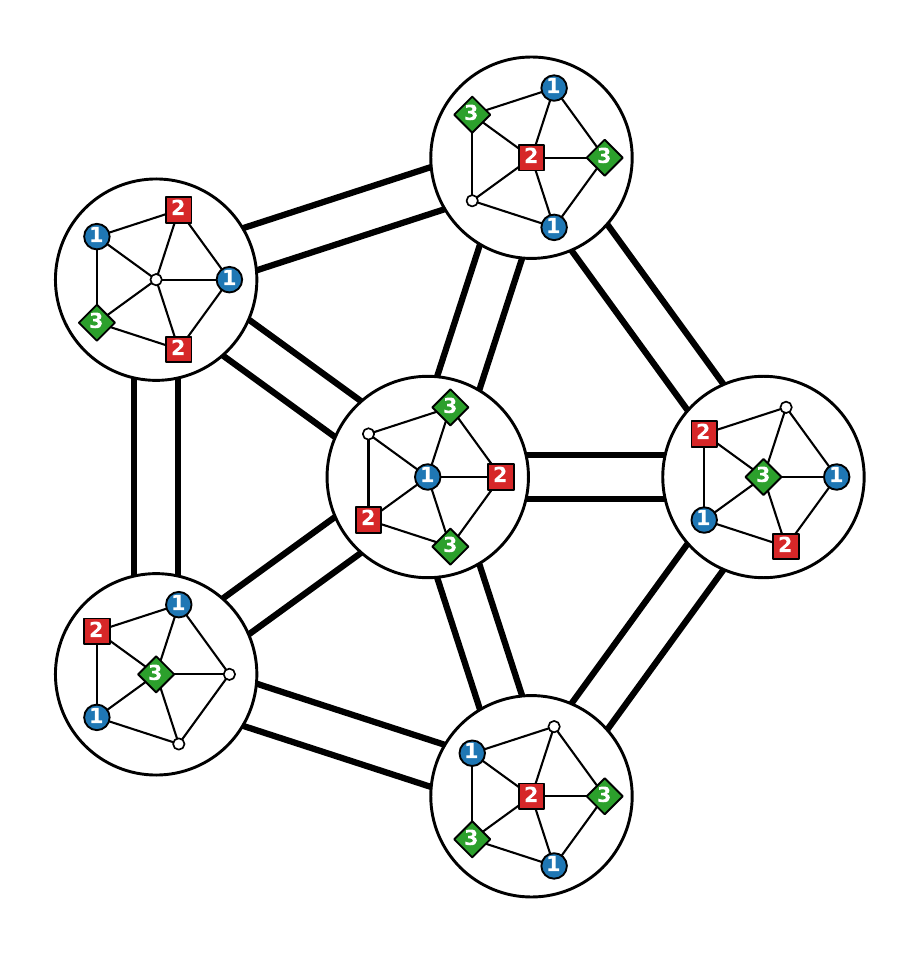}
    \caption{A maximum independent set, which has size $29$, in $W^{\Box 2}_{5}\Box K_3$}
    \label{fig:W^2_5}
\end{figure}

\begin{figure}[!h]
    \centering
    \includegraphics[scale = 0.75]{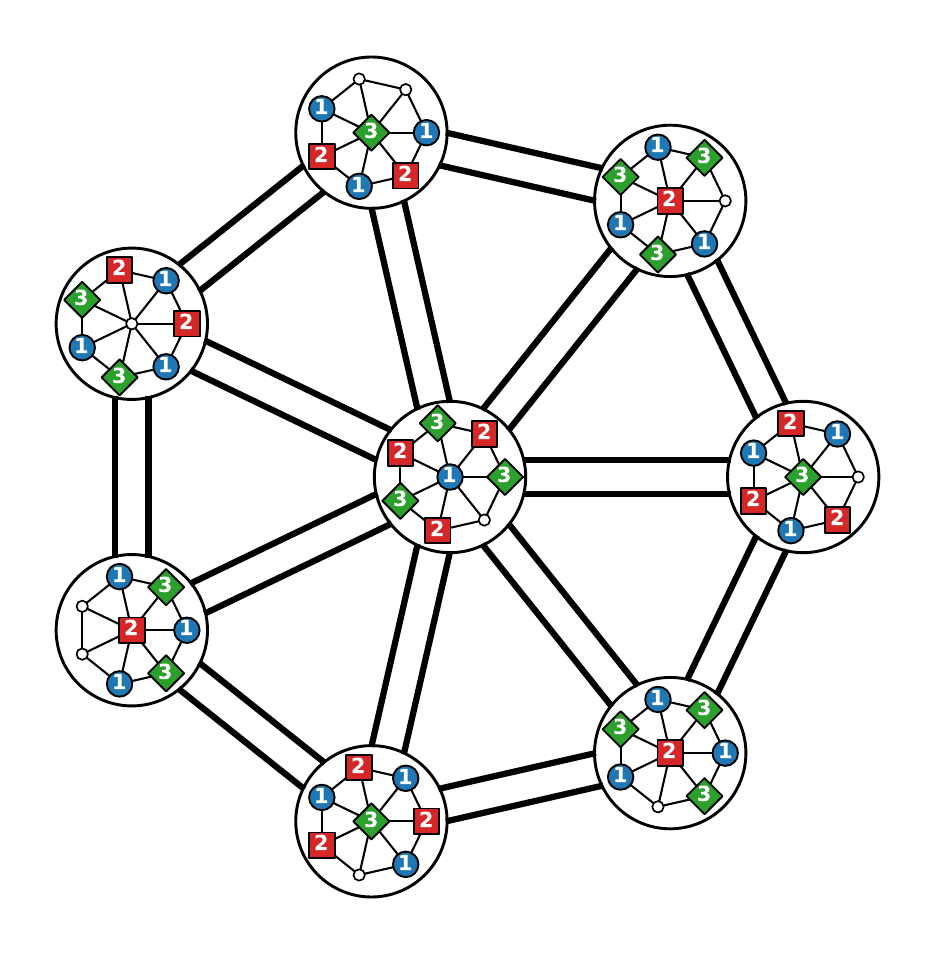}
    \caption{A maximum independent set, which has size $54$,
    in $W^{\Box 2}_{7}\Box K_3$}
    \label{fig:W^2_7}
\end{figure}

\begin{figure}[!h]
    \centering
    \includegraphics[scale = 0.75]{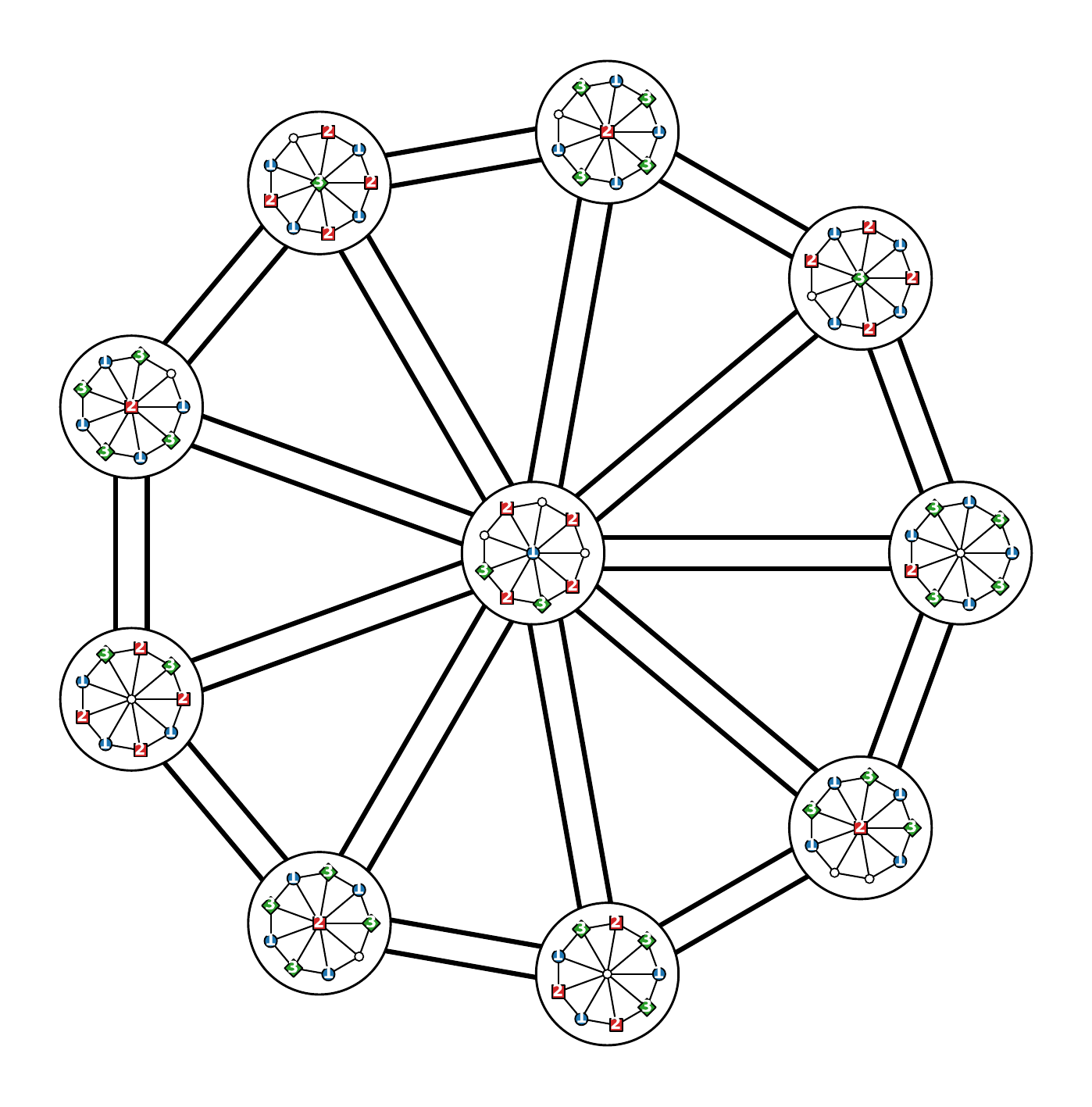}
    \caption{A maximum independent set, which has size $87$,
    in $W^{\Box 2}_{9}\Box K_3$}
    \label{fig:W^2_9}
\end{figure}

\begin{figure}[!h]
    \centering
    \includegraphics[scale = 0.75]{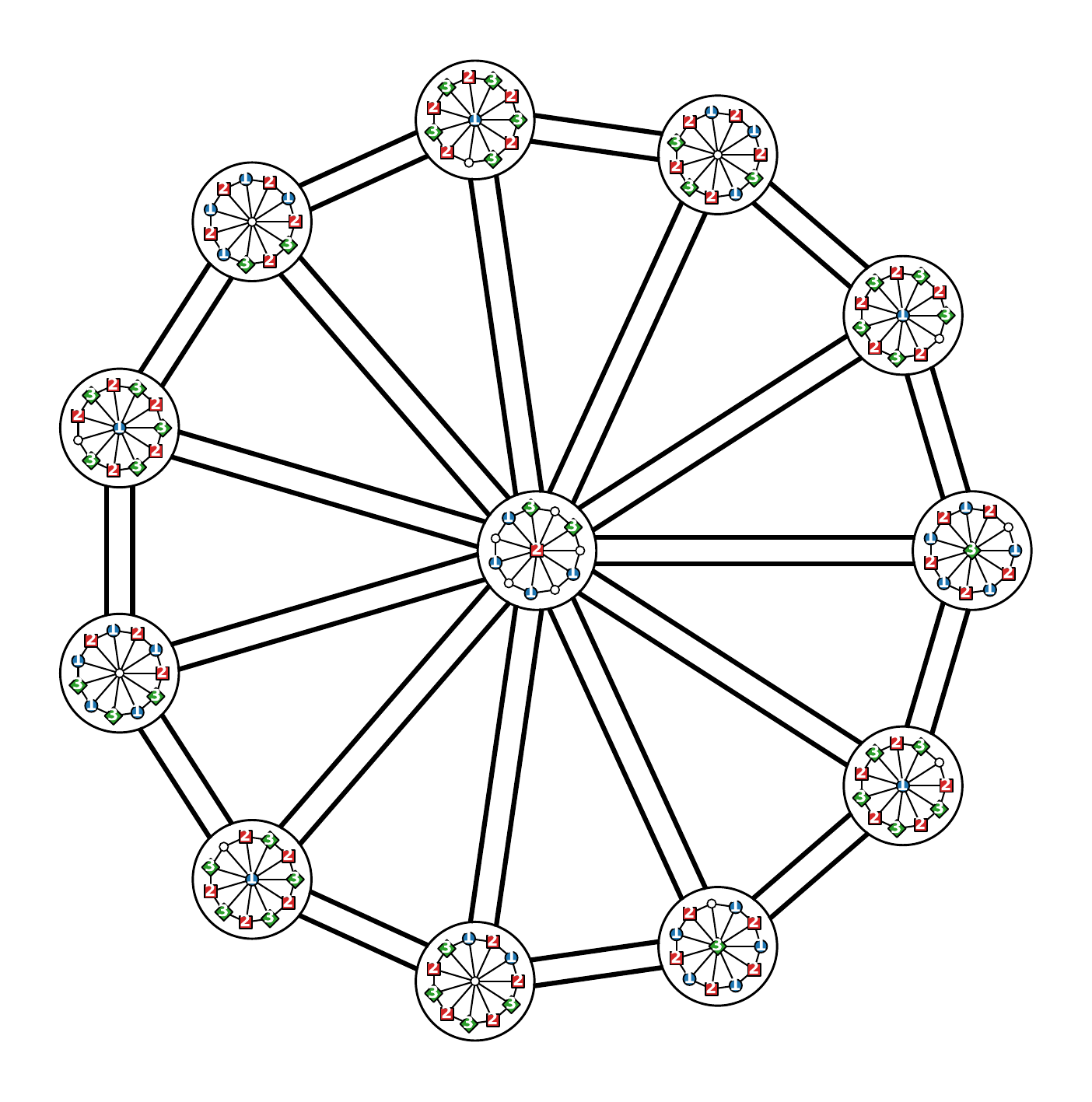}
    \caption{A maximum independent set, which has size $128$,
    in $W^{\Box 2}_{11}\Box K_3$}
    \label{fig:W^2_11}
\end{figure}

\end{document}